\documentclass{amsart}

\usepackage{amssymb}

\usepackage{graphicx}

\usepackage[cmtip,all]{xy}

\newtheorem{theorem}{Theorem}[section]
\newtheorem{lemma}[theorem]{Lemma}

\newcommand{\tuborg}{\left\{\begin{array}{ll}}
\newcommand{\sluttuborg}{\end{array}\right.}
\newcommand{\calO}{\mathcal{O}}

\newcommand{\calI}{\mathcal{I}}
\newcommand{\calZ}{\mathcal{Z}}

\newcommand{\calB}{\mathcal{B}}

\newcommand{\calR}{\mathcal{R}}

\newcommand{\bbZ}{\mathbb{Z}}
\newcommand{\bbP}{\mathbb{P}}
\newcommand{\bbQ}{\mathbb{Q}}

\newcommand{\bbA}{\mathbb{A}}

\newcommand{\bbX}{\mathbb{X}}

\newcommand{\supp}{{\rm Supp}}

\newcommand{\spec}{{\rm Spec}}

\newcommand{\kom}{{\rm Kom}}

\newcommand{\sgn}{{\rm sgn}}

\newcommand{\perm}{{\mathbb{P}\rm erm}}
\newcommand{\edim}{{\rm edim}}

\theoremstyle{definition}
\newtheorem{definition}[theorem]{Definition}

\newtheorem{proposition}[theorem]{Proposition}
\newtheorem{corollary}[theorem]{Corollary}

\newtheorem{question}[theorem]{Question}

\theoremstyle{remark}
\newtheorem{remark}[theorem]{Remark}

\numberwithin{equation}{section}

\begin{document}

\title[Additive Chow groups]{A Hochschild-cyclic approach to additive higher Chow cycles}


\author{Jinhyun Park}
\address{Department of Mathematics, Purdue University\\
West Lafayette, Indiana 47907, USA}
\email{jinhyun@math.purdue.edu; park.jinhyun@gmail.com}
\thanks{This research was partially supported by Purdue University in USA and Institut des Hautes Etudes Scientifiques in France. I would like to thank both the institutions for their supports}



\date{February 5, 2008}



\begin{abstract}Over a field of characteristic zero, we introduce two motivic operations on additive higher Chow cycles: analogues of the Connes boundary $B$ operator and the shuffle product on Hochschild complexes. The former allows us to apply the formalism of mixed complexes to additive Chow complexes building a bridge between additive higher Chow theory and additive $K$-theory. The latter induces a wedge product on additive Chow groups for which we show that the Connes operator is a graded derivation for the wedge product using a variation of a Totaro's cycle. Hence, the additive higher Chow groups with the wedge product and the Connes operator form a commutative differential graded algebra. On zero-cycles, they induce the wedge product and the exterior derivation on the absolute K\"ahler differentials, answering a question of S. Bloch and H. Esnault.
\end{abstract}

\maketitle

\section*{Introduction} There have been so far at least two problems with which the additive higher Chow theory, defined first in \cite{BE1, BE2}, had been shown to have connections with: motives over  $k[x]/(x^m)$ and the Euclidean scissors congruence (see \cite{C, G, P1, P2}). This paper provides the third such area, in particular, regarding the adjective ``additive" that originally comes from the additive $K$-theory. 

Throughout this paper, we always suppose that $k$ is a field of characteristic zero. Under this assumption the additive $K$-theory is the cyclic homology of A. Connes (\cite{FT, LQ}) so that this third interpretation in this paper should involve cyclic objects, or more generally mixed complexes (\emph{c.f.} \cite{L}) and their related formalisms. This paper shows how the additive Chow cycles form such objects with a remark that it could be nice if we can rename them as \emph{Hochschild} Chow cycles, or maybe \emph{additive noncommutative} Chow cycles to reserve the name \emph{additive} Chow cycles for a new cycle complex theory.

In \S 1, we review some relevant properties of $K$-theories: Quillen $K$-theory, Hochschild homology, and cyclic homology. The author explains briefly why he proposes the above new nomenclature, although he still used the name \emph{additive} in this paper.

The \S 2 extends the definition in \cite{P1} of the additive Chow complexes in such a way that for each Artin local $k$-algebra  $(A, \mathfrak{m})$ with $A/ \mathfrak{m} \simeq k$, we associate a cycle complex $\mathcal{Z}^* (X \times \Diamond_* (A))$ for a $k$-variety $X$ (see Definition \ref{additive Chow complex for Artin}). When $A = k[x]/ (x^m)$, we recover the additive Chow groups with modulus considered in \cite{P2, R}. This construction is necessary for our discussion of the shuffle product structure defined in \S 5.

Although it is not used in the sequel we show that it is covariant functorial in $A$ so that it forms a functor $\mathcal{Z}^* (X \times \Diamond_* ( - ) ): (Art/k) \to Kom^- (Ab)$. Combined with the homology functor at $\mathcal{Z}^q (X \times \Diamond_n ( - ))$, we deduce a covariant functor
$$ACH^q (X, n ; - ) : (Art/k) \to (Ab).$$ Thus, the additive Chow theory gives a deformation functor. (\emph{c.f.} \cite{Sch})

The \S 3 and beyond discuss the Hochschild-cyclic argument on additive higher Chow cycles. The \S 3 defines a motivic analogue of the Connes boundary operator $B$ (\emph{c.f.} \cite{L}) denoted by $\delta$ on the additive Chow complexes for $A = k[x] / (x^m)$:
$$\delta: \mathcal{Z}_p (X, n; A) \to \mathcal{Z}_p (X, n+1; A).$$ Unfortunately this Connes operator $\delta$ doesn't behave well with the intersection boundary map $\partial$ on cycles. After using a moving lemma argument as in \cite{Bl2}, we obtain a quasi-isomorphic subcomplex
$$\left( (\mathcal{Z}_* (X, *; A)_0 , \partial ' = \partial_{\rm last} ^{0}) \subset (\mathcal{Z}_* (X, *; A), \partial) \right).$$ We call it the \emph{reduced} additive higher Chow complex. The operator $\delta$  descends to this reduced one, and works well with the boundary map:
\begin{proposition} $\partial ' \delta = \delta \partial'$.
\end{proposition}

For $\mathcal{Z}(n):= \bigoplus_{p \geq 0} \mathcal{Z}_p (X, n; A)_0$, thus we have
\begin{theorem} The reduced additive higher Chow complex $(\mathcal{Z}(*), \partial')$ is a mixed complex with a Connes operator $\delta$:
$$
\xymatrix{\vdots \ar[d] _{\partial'} & \vdots \ar[d] _{\partial'} & \vdots \ar[d]_{\partial'} & \vdots \ar[d]_{\partial'}  \\
\calZ (3) \ar[d] _{\partial'} & \calZ(2) \ar[l] _{\delta} \ar[d] _{\partial'} & \calZ (1) \ar[l]_{\delta} \ar[d] _{\partial'} & \calZ (0) \ar[l]_{\delta} \\
\calZ (2) \ar[d] _{\partial'} & \calZ (1) \ar[l]_{\delta} \ar[d]_{\partial'} & \calZ (0) \ar[l]_{\delta} & \\
\calZ(1) \ar[d] _{\partial'} & \calZ (0)\ar[l]_{\delta} & & \\
\calZ(0) & & &}
$$

\end{theorem}

The total complex will be called the cyclic higher Chow complex, and its homology groups will be denoted by $CCH_p (X, n;A)$, the cyclic higher Chow groups. The usual formalism of mixed complexes then give
\begin{theorem}$ACH$ and $CCH$ form the following Connes periodicity long exact sequence
$$\cdots \overset{B}{\to} ACH_p (n) \overset{I}{\to} CCH_p (n) \overset{S}{\to} CCH_{p-1} (n-2) \overset{B}{\to} ACH_{p-1} (n-1) \overset{I}{\to} \cdots,$$ where $ACH_p (n) := ACH_p (X, n; A)$ and $CCH_p (n) := CCH_p (X, n;A)$. The maps $I, S, B$ have bidegrees $(0,0), (-1, -2), (0, +1)$ in $(p,n)$, respectively. 
\end{theorem}

This result with the calculation $CCH_0 (k, n; k[x]/ (x^2)) \simeq \Omega_{k/\mathbb{Z}}^{n-1}/ d\Omega_{k/\mathbb{Z}} ^{n-2}$ (see Theorem \ref{CCH_0})  supports the author's point regarding the usage of the word \emph{additive}.

Suppose $X= \spec(k)$. In \S 5,  under the identification $\mathbb{A}^{e_1} \times \square^{r_1} \times \mathbb{A}^{e_2} \times \square^{r_2} \simeq \mathbb{A}^{e_1 + e_2} \times \square^{r_1 + r_2}$ we first define the concatenation $\times$ and the shuffle product $\times _{sh}$ of additive Chow cycles. Specifically for given two Artin local $k$-algebras $A_1$ and $A_2$, we have
$$\times , \times_{sh} : \mathcal{Z}_p (k, r_1 ; A_1) \otimes \mathcal{Z}_q (k, r_2 ; A_2) \to \mathcal{Z}_{d} (k, n ; A_1 \otimes _k A_2),$$ where $d:= p + q$ and $n := r_1 + r_2$. The intersection boundary $\partial$ is a graded derivation for $\times$ and $\times_{sh}$ so that for $\times_\cdot = \times$ or $\times_{sh}$, we have
$$\partial (x \times_{\cdot} y ) = (\partial x) \times_{\cdot} y + (-1)^{r_1} x \times_{\cdot} (\partial y).$$ These product maps induce the homomorphisms on the additive Chow groups:
$$\times_{{\cdot} *} : ACH_p (k, r_1 ; A_1) \otimes ACH_q (k, r_2 ; A_2) \to ACH_{d} (k, n; A_1 \otimes_k A_2),$$ where $\times_{\cdot *} = \times_*$ or ${sh *}$.

When $A _1= k[x]/ (x^{m_1})$ and $A_2 = k[x] / (x^{m_2})$, using the product $\mu: \mathbb{G}_m \times \mathbb{G}_m \to \mathbb{G}_m$, we obtain a well-defined multiplication $\mu_*$ on the image of $sh_*$  in $\mathcal{Z}_{d} (k, n ; A_1 \otimes_k A_2)$. We define the wedge product $\wedge:= \mu_* \circ sh_*$ that gives a homomorphism
$$\wedge: ACH_p (k, r_1 ; k[x]/(x^{m_1}) )\otimes ACH_q (k, r_2 ; k[x] / (x^{m_2})) \to ACH_{d} (k, n ; k[x] / (x^m)),$$ where $m = \min \{ m_1, m_2 \}$.

When $A:=A_1 = A_2 = k[x] / (x^m)$, the most interesting result between the Connes map $\delta_*$ and the wedge product $\wedge$ is the following:
\begin{theorem}For cycles $\xi \in \mathcal{Z}_p (k, r_1; A)_0, \eta \in \mathcal{Z}_q (k, r_2 ; A)_0$ with $\partial' \xi = 0$, $\partial' \eta = 0$, we have
$$\delta_* (\xi \wedge \eta) - (\delta_* \xi ) \wedge \eta - (-1)^{r_1} \xi \wedge (\delta_* \eta) = - \partial (\xi \wedge' \eta),$$ for some cycle $\xi \wedge' \eta$ (see Definition \ref{definition cyclic shuffle product}), called the cyclic shuffle product. That is, the Connes map $\delta_*$ is a graded derivation for the wedge product $\wedge$ on the additive higher Chow groups.
\end{theorem}

\begin{corollary}The triple $(ACH_* (k, *; A), \wedge, \delta_*)$ is a CDGA. In particular, when $A= k[x]/(x^2)$ and for $0$-cycles, the CDGA $(\Omega_{k/\mathbb{Z}} ^*, \wedge, d)$ is motivic. 
\end{corollary}

This answers the Challenge 5.3 of Bloch and Esnault in \cite{BE1} that asks for a motivic description of the wedge product and the exterior product on the K\"ahler differentials.

\section{Additive Chow theory and additive $K$-theory}
Several results on various $K$-theories presented in this section guide our studies, where we use a paradigm in \cite{L2} on the classification of various $K$-theories:

\begin{table}[h]
\begin{center}
\begin{tabular}{|c|c|c|}\hline
 & noncommutative & commutative \\
 \hline
 multiplicative & Leibniz $K$-theory? & Quillen $K$-theory \\
 \hline
 additive & Hochschild homology & cyclic homology\\
\hline
\end{tabular}
\end{center}
\end{table}

For instance, we think of the Hochschild homology theory as the additive noncommutative $K$-theory. We will compare the Quillen $K$-theory, Hochschild homology, and cyclic homology. First, recall the following result:

\begin{theorem}[\cite{NS, S, T}]\label{Quillen}The group homology of $GL (k)$ has the following list of properties: 
\begin{enumerate}
\item [Q1)] stability:
$$
H_n (GL_n (k) ;\bbQ) \overset{\simeq}{\to} H_n (GL_{n+1}(k); \bbQ) \overset{\simeq}{\to} \cdots \overset{\simeq}{\to} H_n (GL (k); \bbQ)
$$
\item [Q2)] 1st obstruction to stability: The following sequence is exact:
$$
H_n (GL_{n-1} (k); \bbQ) \to H_n (GL_n (k); \bbQ) \to K_n ^M (k) \to 0
$$
\item [Q3)] primitive part: $H_* (GL(k), \bbQ)$ is a graded Hopf algebra with its primitive part
$$
Prim_* (H_* (GL(k), \bbQ)) \simeq K_* (k)_{\bbQ}.
$$
\item [Q4)] $0$-cycles: The higher Chow groups give the obstruction groups in {\rm Q2)}:
$$CH_0 (k, n) \simeq K_n ^M (k).$$
\end{enumerate}
\end{theorem}

\bigskip

When we work with the Lie algebra $gl (k)$ and its Lie algebra homology $H_n (gl(k))$ instead of the group $GL(k)$ and its group homology, we get \emph{the additive $K$-theory} $K_n ^+ (k)$. This group $K_n ^+ (k) \simeq HC_{n-1} (k)$, thus the additive $K$-theory is in fact the cyclic homology, with a shift of the degree. This theory too enjoys analogous properties:

\begin{theorem}[\cite{FT, L, LQ}]\label{cyclic homology}The Lie algebra homology of $gl(k)$ has the following list of properties:
\begin{enumerate}
\item [C1)] stability:
$$H_n (gl_n (k); k) \overset{\simeq}{\to} H_n (gl_{n+1} (k); k) \overset{\simeq}{\to} \cdots \overset{\simeq}{\to} H_n (gl(k); k)
$$
\item [C2)] 1st obstruction to stability: The following sequence is exact:
$$
H_n (gl_{n-1} (k); k) \to H_n (gl_n (k); k) \to \Omega_{k/\bbQ} ^{n-1} / d \Omega_{k/\bbQ} ^{n-2} \to 0.$$
\item [C3)] primitive part: $H_* (gl(k); k)$ is a Hopf algebra with its primitive part
$$Prim_* (H_* (gl(k); k)) \simeq K^+ _* (k)$$ the additive $K$-theory of Feigin and Tsygan, which is isomorphic to the cyclic homology $HC_{*-1} (k)$ over $\bbQ$.
\end{enumerate}

\end{theorem}

From the comparison of the Quillen $K$-theory (Theorem \ref{Quillen}) and the additive $K$-theory (Theorem \ref{cyclic homology}), we can call the group $\Omega_{k/\bbQ} ^{n-1} / d \Omega_{k/\bbQ} ^{n-2}$ the Milnor-cyclic homology $HC_{n-1} ^M (k)$. (\emph{c.f.} 10.3.3 in \cite{L}) It is easy to see that for ${\rm char} (k) = 0$, we have $\Omega_{k/\bbQ} ^{n-1} = \Omega_{k/\bbZ} ^{n-1}$.

One problem we see here is that C4) is missing from Theorem \ref{cyclic homology}, and despite the name \emph{additive} higher Chow group, its zero-cycle group $ACH_0 (k, n-1)$ does not produce the obstruction group $HC_{n-1} ^M (k)$. That means, although the groups $ACH$ are important in the calculation of the motivic cohomology groups of the fat point $\spec \left( k[\epsilon]/ \epsilon^2 \right)$ as seen in \cite{P1},  the additive Chow theory is not quite the right additive cycle theory for the additive $K$-theory. 

But, the Hochschild homology and the theory of Leibniz homology of Cuvier and Loday (\cite{Cu, L}) show that the naming of $ACH$ was very close; the Leibniz homology is a noncommutative vesion of the Lie algebra homology in the sense that we systematically ignore the antisymmetry axiom of Lie algebras, where we replace the Jacobi identity by the Leibniz rule, and we replace the wedge product $\wedge$ by the tensor product $\otimes$ in the Chevalley-Eilenberg resolution of the Lie algebra. (See p. 302 and p. 326 in \cite{Cu})

\begin{theorem}[\cite{BE2, Cu, L}]\label{Hochschild homology}The Leibniz homology $HL_* (gl(k))$ of $gl(k)$ has the following list of properties:
\begin{enumerate}
\item [H1)] stability:
$$
HL_n (gl_n (k)) \overset{\simeq}{\to} HL_n (gl_{n+1} (k)) \overset{\simeq}{\to} \cdots \overset{\simeq}{\to} HL_n (gl (k))
$$
\item [H2)] 1st obstruction to stability: The following sequence is exact:
$$
HL_n (gl_{n-1} (k)) \to HL_n (gl_n (k)) \to \Omega_{k/\bbQ} ^{n-1} \to 0.$$
\item [H3)] primitive part: $HL_* (gl(k))$ is a Hopf algebra with its primitive part
$$
Prim_* (HL_* (gl(k))) \simeq HH_{*-1} (k),
$$ the Hochschild homology of $k$ over $\bbQ$.
\item [H4)] $0$-cycles: The additive higher Chow groups give the obstruction groups in {\rm H2)}:
$$ACH_0 (k, n-1) \simeq \Omega_{k/\bbQ} ^{n-1}.$$
\end{enumerate}
\end{theorem}

We can call the group $\Omega_{k/\bbQ} ^{n-1}$ the Milnor-Hochschild homology group $HH_{n-1} ^M (k)$ of $k$ over $\bbQ$ using the analogy between Theorem \ref{Quillen} and Theorem \ref{Hochschild homology} as before. (\emph{c.f.} p. 337 in \cite{L})

\bigskip

Thus, we may suspect that the group $ACH$ should actually be called the additive \emph{noncommutative} higher Chow groups. Since the Hochschild homology $HH_n$ and the cyclic homology $HC_n$ satisfy the Connes periodicity sequence
$$\cdots \to HH_n (k) \overset{I}{\to} HC_n (k) \overset{S}{\to} HC_{n-2} (k) \overset{B}{\to} HH_{n-1} (k) \to \cdots,$$ the absence of C4) suggests the following new question:

\begin{question}Can we find a cycle complex, to be named as the \emph{additive commutative higher Chow complex} or \emph{cyclic higher Chow complex}, that satisfies at least the following two properties:
\begin{enumerate}
\item [(i)] its homology fits into a Connes periodicity exact sequence as the cyclic part, where the $ACH$ fits into the sequence as the Hochschild part.
\item [(ii)] its zero-cycle group gives the right obstruction group of stability: $\Omega_{k/\bbZ}^{n-1}/ d \Omega_{k/\bbZ} ^{n-2}$.
\end{enumerate}
\end{question}

The answer to this question is given in the Theorems \ref{Connes periodicity} and \ref{CCH_0}.

\section{Additive Chow groups associated to Artin local rings}

Let $q, n \geq 0$ be integers, and let $X$ be an equi-dimensional quasi-projective $k$-variety. Let $(Art/k)$ be the category of Artin local $k$-algebras with the residue field $k$. For an object $A$ in $(Art/k)$, we denote by $\mathfrak{m}_A$ its unique maximal ideal. For two rings $A, B \in (Art/k)$, a morphism $A \to B$ is a $k$-algebra homomorphism such that $f(\mathfrak{m}_A) = \mathfrak{m}_B$. For each integer $e \geq 0$, denote by $\mathbb{X}_e$, the set of $l$ indeterminates $\{ x_1, \cdots, x_e \}$.

The section extends the definition of the additive higher Chow groups with modulus in \cite{P1, R} to the category $(Art/k)$. In this way, we obtain a functor
$$ACH^q (X, n; -) : (Art/k) \to (Ab),$$ where $(Ab)$ is the category of abelian groups. When the Artin local ring $A$ is $k[x]/(x^m)$ where $m \geq 2$, we recover the notion of additive Chow groups with modulus. This functor $ACH^q (X, n; - )$ is a \emph{functor of Artin rings} in the sense of Schlessinger \cite{Sch}, which is sometimes called a \emph{deformation functor}.

The essential point of the construction is to use the embedding dimension (see Definition \ref{embedding dimension}) of a given Artin ring $A$ to construct an affine space on which our algebraic cycles reside. 

\subsection{Presentations of Artin local $k$-algebras}

Let $(A, \mathfrak{m}_A)$ be an Artin ring in $(Art/k)$.
\begin{definition}\label{embedding dimension} The \emph{embedding dimension} of $A$ is the dimension of the $k$-vector space $\mathfrak{m}_A/ \mathfrak{m}_A ^2$. We denote it by $\edim (A)$.
\end{definition}

The following is a polynomial version of the Cohen structure theorem for Artin local $k$-algebras.

\begin{lemma}\label{presentation}Let $(A, \mathfrak{m}_A)$ be an Artin ring in $(Art/k)$ with embedding dimension $e$. Then there is a polynomial presentation
$$0 \to J \to k[\bbX_e] \to A \to 0$$ such that the ideal $J$ of $k[\bbX_e]$ is contained in the ideal $\mathfrak{m}^2$, where $\mathfrak{m}$ is the ideal generated by $\bbX_e$ in $k[\bbX_e]$. This $e$ is the minimal number of indeterminates for which $A$ has a polynomial presentation.
\end{lemma}
\begin{proof}Our proof uses the Cohen's theorem applied to $A$, that already gives us a presentation by a formal power series ring
\begin{eqnarray}\label{formal presentation}0 \to \widehat{J} \to k [[\bbX_e]] \to A \to 0\end{eqnarray} where $\widehat{\mathfrak{m}}$ is the ideal generated by $\bbX_e$ in $k[[\bbX_e]]$ and $\widehat{J}$ is an ideal in $k[[\bbX_e]]$ satisfying $\widehat{J} \subset \widehat{\mathfrak{m}}^2$. We choose a presentation with the minimal value of $e$.

Since $A$ is Artinian, there is an integer $N$ such that $\widehat{\mathfrak{m}}^N \subset \widehat{J} \subset \widehat{\mathfrak{m}}^2$. Let $I_N$ be the image of $\widehat{J}$ in $k[\bbX_e]/\mathfrak{m}^N$ under the isomorphism of $k$-algebras
$$ k[[\bbX_e]]/ \widehat{\mathfrak{m}}^N \overset{\sim}{\to} k[\bbX_e]/\mathfrak{m} ^N.$$ Then the sequence \eqref{formal presentation} gives an exact sequence
$$0 \to I_N \to k[\bbX_e]/\mathfrak{m}^N \to A\to 0.$$ Let $J$ be the ideal in $k[\bbX_e]$ whose image in $k[\bbX_e]/\mathfrak{m}^N$ is $I_N$, thus we obtain a commutative diagram
$$\xymatrix{ 0 \ar[r] & I_N \ar[r] & k[\bbX_e] /\mathfrak{m}^N \ar[r] & A \ar[r] & 0 \\ 0 \ar[r] & J \ar[u] \ar[r] & k[\bbX_e] \ar[u] & &}$$ that immediately yields a desired presentation. That $e$ is a minimal such integer is obvious.
\end{proof}

Let $PresFin(k)$ be a category defined in the following fashion. Its objects are pairs $\left( k[\bbX_e], J \right)$, where $J$ is an ideal of the polynomial ring $k[\bbX_e]$ such that $k[\bbX_e] / J \in (Art/k)$. For two objects $\left(k[\bbX_{e_1}], J_1\right)$ and $\left( k[\bbX_{e_2}], J_2 \right)$, and two $k$-algebra homomorphisms $\phi, \psi: k[\bbX_{e_1}] \to k[\bbX_{e_2}]$ such that $\phi(J_1) \subset J_2$, $\psi(J_1) \subset J_2$, we define an equivalence $\phi \sim \psi$ if the induced homomorphisms $\bar{\phi}, \bar{\psi}: k[\bbX_{e_1}]/J_1 \to k[\bbX_{e_2}]/J_2$ are equal. Morphisms in $PresFin(k)$ are the collection of all such $k$-algebra homomorphisms modulo the above equivalence.

\begin{lemma}Let $A_1, A_2$ be two Artin local $k$-algebras in $(Art/k)$ with presentations $P_1=\left(k[\bbX_{e_1}], J_1 \right)$, $P_2=(\left( k[\bbX_{e_2}], J_2 \right)$ given by the Lemma \ref{presentation}. Let $f: A_1 \to A_2$ be a morphism in $(Art/k)$. Then, it induces a unique morphism $\tilde{f}: P_1 \to P_2$ in $PresFin(k)$. Furthermore, there is a natural bijection
$$\hom_{(Art/k)} (A_1, A_2) \overset{\sim}{\to} \hom_{PresFin(k)} \left( P_1, P_2 \right).$$
\end{lemma}

\begin{proof}
We construct $\widetilde{f}$ that fits into the diagram:
$$\xymatrix{ 0 \ar[r] & J_1 \ar[d] \ar[r] & k[\mathbb{X}_{e_1}] \ar[r]^{q_1} \ar[d] ^{\widetilde{f}} & A_1 \ar[r] \ar[d] ^f & 0 \\
0 \ar[r] & J_2 \ar[r] & k[\mathbb{X}_{e_2}] \ar[r] ^{q_2} & A_2 \ar[r] & 0}$$

This is easy: for each $i \in \{ 1, \cdots, e \}$, pick $y_i \in q_2 ^{-1} f(q_1 (x_i))$ and define $\widetilde{f} (x_i) = y_i$. It determines $\widetilde{f}$. For $ g \in J_1$, since $q_1 (g) = 0$ we have $\widetilde{f} (g) \in q_2 ^{-1} (0) = J_2$. Suppose we choose a different set of $y_i ' \in q_2 ^{-1} (f (q_1 (x_i)))$ and the corresponding $\widetilde{f}'$. Then, $y_i - y_i ' \in J_2$ so that $q_2 (y_i ) = q_2 (y_i ') = f(q_1 (x_i))$. Hence $\widetilde{f}$ and $\widetilde{f}'$ induce the same map $f: A_1 \to A_2$. Hence in the category $PresFin (k)$, they are the equal morphisms. 

The bijection between the two sets obvious.\end{proof}

\begin{lemma}For two objects $P_i = (k[\mathbb{X}_{e_i}], J_i)$, $i = 1, 2$, if there is an isomorphism $P_1 \simeq P_2$ in $PresFin (k)$, then the corresponding Artin rings $A_1: =k[\mathbb{X}_{e_1}]/J_1$ and $ A_2: = k[\mathbb{X}_{e_2}]/J_2$ are isomorphic under the induced isomorphism.
\end{lemma}
\begin{proof}
Suppose that $\widetilde{g} \circ \widetilde{f} = {\rm Id}_{P_1}$ and $\widetilde{f} \circ \widetilde{g} = {\rm Id}_{P_2}$ for morphisms $\widetilde{f}: P_1 \to P_2$ and $\widetilde{g}: P_2 \to P_1$. The maps $\widetilde{f}$ and $\widetilde{g}$ induce maps $f: A_1 \to A_2$ and $g: A_2 \to A_1$. That $\widetilde{g} \circ \widetilde{f} = {\rm Id}_{P_1}$ means $g \circ f = {\rm Id}_{A_1}$. Similarly we hae $f \circ g = {\rm Id} _{A_2}$. Thus $A_1 \simeq A_2$.\end{proof}

Thus, we have an equivalence of categories
\begin{eqnarray}\label{fully faiththful}\iota: (Art/k) \to PresFin(k)\end{eqnarray}
that sends an Artin local $k$-algebra $A$ to its presentation $P$, and a morphism $f: A_1 \to A_2$ of Artin local $k$-algebras to a morphism $\widetilde{f} : P_1 \to P_2$ defined as above.

\subsection{Additive higher Chow complex}Let $X$ be an irreducible quasi-projective variety over a field $k$, and let $p, q, n\geq 0 $ be integers. 
Let $\square= \bbP^1 - \{ 1 \}$, and let $P = \left(\bbX_e, J\right)$ be an object in $PresFin(k)$. Define an affine space
$$\Diamond_n (P) := \bbA^e \times \square^n.$$ Let $\widehat{\Diamond}_n (P) = \bbA^e \times \left( \bbP^1 \right)^n$. For an irreducible closed subvariety $Z \subset X \times \Diamond_n (P)$, let $\widehat{Z}$ be its Zariski closure in $X \times \widehat{\Diamond}_n (P)$ and let $\nu:\overline{Z} \to \widehat{Z}$ be its normalization. Let $\calI (P)$ be the sheaf of ideals of $\mathcal{O}_{\widehat{\Diamond}_n (P)}$ generated by $\mathbb{X}_e$. Let $V_0:= V (\calI(P))$ be the closed subvariety of $X \times \widehat{\Diamond}_n (P)$ associated to the ideal sheaf $\calI (P)$. For the Artin ring $A:= k[\bbX_e]/J$, we also denote then by $\calI (A)$ and $V (\calI (A))$.

For each $i \in \{ 1, \cdots, n\}$ and $j \in \{ 0, \infty\}$ we have codimension one ``nondegenerate'' faces \begin{eqnarray*}\tuborg F_i ^j : X \times \Diamond_{n-1} \hookrightarrow  X \times \Diamond_n  \\ (y, x, t_1, \cdots, t_{n-1})\mapsto  (y, x, t_1, \cdots, t_{i-1}, j, t_{i}, \cdots, t_{n-1}) \sluttuborg \end{eqnarray*} 
and the ``degenerate'' face
\begin{eqnarray*}V_0 \hookrightarrow X \times \Diamond_n .\end{eqnarray*}Various higher codimension faces are obtained by intersecting the above faces as well.

We associate to each object $P= \left(k[\bbX_d], J\right)$ a complex of abelian groups $\calZ^q (X \times \Diamond_* (P)) \in \kom^- (Ab)$ as follows:

\begin{definition}\label{additive Chow complex for Artin} Let $c_0 (X, n; P)$ be the free abelian group on the set of $0$-dimensional irreducible reduced closed points in $X \times \Diamond_n (P)$ not intersecting the faces. For $p \geq 1$, let $c_p (X, n ; P)$ be the free abelian group on the set of $p$-dimensional irreducible closed subvarieties $Z \subset X \times \Diamond_n (P)$ satisfying
\begin{enumerate}
\item $\widehat{Z}$ intersects all lower dimensional faces properly.
\item For each closed point $\mathfrak{p} \in \nu^* V_0$, there is an integer $1 \leq i \leq n$ such that
$$1-t_i \in (J) \cdot \calO_{\overline{Z}, p}.$$
\end{enumerate}

This second condition is called the \emph{modulus condition} in $t_i$. Let $\calZ_p (X, n; P)$ be the group $c_p (X, n ; P)$ modulo the subgroup of degenerate cycles, \emph{i.e.} those obtained by pulling back cycles on $X \times \Diamond_{n-1} (P)$ via various projections. The cycles in $\calZ_* (X, *; P)$ are said to be admissible.

For $i \in \{ 1, \cdots, n \}$ and $j \in \{ 0, \infty \}$, let $\partial _i : \calZ_p (X, n  ; P) \to \calZ_{p-1} (X, n-1 ; P)$ be the intersection product with the face $F_i ^j$. Let $\partial := \sum_{i=1} ^n (-1)^i \left(\partial _i ^0 - \partial _i ^{\infty} \right)$. It is easy to see that $\partial ^2 =0$. Thus, we have a complex
$$\cdots \overset{\partial}{\to} \calZ_{p+1} (X \times \Diamond_{n+1} (P)) \overset{\partial}{\to} \calZ_p (X \times \Diamond_n (P)) \overset{\partial}{\to} \calZ_{p-1} (X \times \Diamond_{n-1}( P)) \overset{\partial}{\to} \cdots$$ called the additive Chow complex associated to $P$. If $p + q = \dim X + e + n$, then we also write $\calZ^q (X \times \Diamond_n (P))$ in terms of codimensions. 
\end{definition}

\begin{definition} The homology group at $\calZ_p (X, n ; P)$, denoted by $ACH_p (X, n; P)$, is called the \emph{additive Chow group associated to $P$.} For an Artin local $k$-algebra in $(Art/k)$, choose a presentation $P$ for $A$, and define $ACH_p (X, n; A):= ACH_p (X, n ; P)$. If $p + q = \dim X + e + n$, then we define $ACH^q (X, n; P) = ACH_p (X, n;P)$. \end{definition}
\begin{lemma}Suppose we have two isomorphic objects $P_1 \simeq P_2$ in $PresFin(k)$. Then, the complex $\left( \calZ^q (X \times \Diamond_*; P_1), \partial \right)$ is isomorphic to $\left( \calZ^q (X \times \Diamond_*; P_2) , \partial \right)$.
\end{lemma}
\begin{proof} Obvious.\end{proof}

\begin{lemma} When $A= k$, the group $\calZ_p (X \times \Diamond_n (k)) = 0$. In particular, its homology group $ACH_p (X, n ; k) = 0$.
\end{lemma}

\begin{proof}The ring $k$ has its embedding dimension $0$, thus its corresponding presentation is $P= (k, 0)$ with $\mathfrak{m}_k =  0$. Thus, the set $V_0$ is the whole $X \times \Diamond_n (k) = X \times \square^n$. Hence if $Z \subset X \times \Diamond_n (k)$ is an admissible irreducible closed subvariety with the normalization $\nu: \overline{Z} \to \widehat{Z}$, then $\nu^* V_0 = \overline{Z}$. But, the modulus condition requires that for all $\mathfrak{p} \in \overline{Z} = \nu^* V_0$ we must have $1- t_i \in ( 0) \cdot \calO_{\overline{Z}, \mathfrak{p}}$, \emph{i.e.} $1 = t_i$ at $mathfrak{p}$. Hence, we must have $\nu (\mathfrak{p}) \in \widehat{Z} \backslash Z$ for all $\mathfrak{p}$. But this is impossible unless $Z= \emptyset$. Thus $\calZ_p (X, n; k)= 0$ so that $ACH_p (X, n; k) = 0$.\end{proof}

\subsection{Functoriality}
In this section, for two objects $P_1, P_2$ of $PresFin(k)$ and a morphism $P_1 \to P_2$, we define a homomorphism $\calZ^q (X \times \Diamond_n (P_1)) \to \calZ^q (X\times \Diamond_ n (P_2))$. Then we show that it gives a functor
$$\calZ^q (X \times \Diamond_{*} (-)) : (Art/k) \to \kom^- (Ab).$$

 Let $A_1, A_2$ be two Artin local $k$-algebras in $(Art/k)$ with embedding dimensions $e_1, e_2$, respectively. We define $\calZ^q (X \times \Diamond_n (f))$ for each $f: A_1 \to A_2$. It will be denoted by $f_*$ when no confusion arises. There corresponds  a $k$-algebra homomorphism $\widetilde{f}: k[\bbX_{e_1}] \to k[\bbX_{e_2}]$ on the level of their minimal presentations. It induces a morphism $f^{\#}: X \times \bbA^{e_2} \times \square ^n \to X \times \bbA^{e_1} \times \square^n$. For an irreducible closed subvariety $Z \subset X \times \bbA^{e_1} \times \square^n$, consider the fibre square
 $$\xymatrix{ Z \times _{X \times \Diamond_n (A_1)} X \times \Diamond_n (A_2) \ar[r] \ar[d] & X \times \Diamond_n (A_2) \ar[d] ^{f ^\#} \\
 Z \ar[r] &X \times \Diamond_n (A_1) }$$ and define $f_* (Z) := \calZ^q (X \times \Diamond_n (f)) (Z) := [ Z \times_{X \times \Diamond_n (A_1) }X \times \Diamond_n (A_2)]$, where $[ \ ]$ means the associated cycle.

\begin{lemma}For $f: A_1 \to A_2$ and an irreducible reduced admissible closed subvariety $Z \in \calZ^q (X \times \Diamond_n (A_1))$, the cycle $f_* (Z)$ is admissible, thus belongs to $ \calZ^q (X \times \Diamond_n (A_2))$.
\end{lemma}
\begin{proof} Choose presentations $P_i = (k [\mathbb{X}_{e_i}], J_i)$ for $A_i$, $i =1 , 2$. Let $W \subset Z \times_{X \times \Diamond_n (A_1)} X \times \Diamond_n (A_2)$ be an irreducible reduced component. Consider a commutative diagram
$$\xymatrix{ \overline{W} \ar[r]^{\nu_1} \ar[d] ^{\phi} & \widehat{W} \ar[d] \ar[r] & X \times \widehat{\Diamond}_n (A_2) \ar[d] ^{f ^{\#}} \\ \overline{Z} \ar[r] ^{\nu} & \widehat{Z} \ar[r] & X \times \widehat{\Diamond}_n (A_1)}$$ where $\widehat{W}$ and $\widehat{Z}$ are the Zariski closures and $\nu_1$ and $\nu$ are normalization maps. The morphism $\phi$ is given by the universal property of the normalization $\nu$. 

To check the modulus condition, let $\mathfrak{p} \in \nu_1 ^* V (\calI (P_2))$ be a closed point. Then, 
\begin{eqnarray*}
\phi (\mathfrak{p}) \in \phi (\nu_1 ^* V (\calI (P_2)) &=& \nu ^* (f ^{\#} (V (\calI (P_2)))) \\
&=& \nu^* V ( \calI (P_1) ).
\end{eqnarray*} Thus, by the modulus condition at $\phi (\mathfrak{p})$, there exists $1 \leq i \leq n$ such that $$1 - t_i \in (J_1) \cdot \mathcal{O}_{ \overline{Z}, \phi (\mathfrak{p}) }.$$ Since we have a natural map $\mathcal{O}_{\overline{Z}, \phi (\mathfrak{p})} \to  \mathcal{O}_{\overline{W}, \mathfrak{p}}$ for which the ideal $J_1$ is mapped into $J_2$, we have a natural map $(J_1) \cdot \calO_{\overline{Z}, \phi (\mathfrak{p})} \to (J_2) \cdot \calO_{\overline{W}, \mathfrak{p}}$. Thus, we have $1- t_i \in (J_2) \cdot \calO_{\overline{W}, \mathfrak{p}}$ as desired. This proves the modulus condition.

Proper intersection with faces is obvious since the map $f^{\#}$ is a closed immersion whose image is at  worst one of the faces. Hence $f_*$ sends admissible cycles to admissible cycles. \end{proof}

Notice that in terms of dimensions, where $p + q = \dim X + e_1 +n$, the induced map $f_*$ maps $\calZ_p (X \times \Diamond_n (A_1))$ into $\calZ_{p + (e_2 - e_1)} (X \times \Diamond_n (A_2))$.

\begin{lemma}For $n \geq 0$, the maps $f_*$ are compatible with the boundary maps. In other words, $f_*: \calZ^q (X \times \Diamond_* (A_1)) \to \calZ^q (X \times \Diamond_* (A_2))$ is a morphism in $\kom^- (Ab)$.
\end{lemma}

\begin{proof}
Obvious.\end{proof}

The last property is the functoriality.

\begin{lemma}Consider two morphisms $A_1 \overset{f_1}{\to} A_2 \overset{f_2}{\to} A_3$ in $(Art/k)$. Then, the induced morphisms of complexes satisfy
$$(f_2 \circ f_1)_* = {f_2}_* \circ {f_1}_*.$$ In other words, we have a covariant functor
$$\calZ^q (X \times \Diamond_{*} ( - )): (Art/k) \to \kom^- (Ab).$$
\end{lemma}

\begin{proof}The transitivity of the fibre products implies the functoriality.\end{proof}

Thus, by composing the above functor with the homology functor $H_n: \kom^- (Ab) \to (Ab)$, we obtain the following corollary:

\begin{corollary}We have the additive higher Chow functor
$$ACH^q (X, n ; - ) : (Art/k) \to (Ab).$$
\end{corollary}

A detailed study of this deformation functor is not pursued in this paper, and should be treated elsewhere.

\section{A motivic Connes boundary operator $\delta$}

\subsection{K\"ahler differentials}

Recall the following useful result on the absolute K\"ahler differentials:

\begin{lemma}[Lemma 4.1 in \cite{BE1}]\label{differential}We have an isomorphism of $k$-vector spaces
$$\phi: \frac{k \otimes \wedge^n k^{\times}}{\calR} \overset{\sim}{\to} \Omega_{k/\bbZ} ^n$$
$$ a \otimes b_1 \wedge \cdots \wedge b_n \mapsto a d \log b_1 \wedge \cdots \wedge d \log b_n,$$ where $\calR$ is the vector subspace spanned by all elements of type
$$ \left( a \otimes a + (1-a) \otimes (1-a) \right) \wedge b_1 \wedge \cdots \wedge b_{n-1}$$ for $a \in k, b_i \in k^{\times}$.
\end{lemma}

\begin{remark}\label{motivation 1}{\rm What do exact forms in $\Omega_{k/\bbZ} ^{n}$ look like? For a generator $a d \log b_1 \wedge \cdots  \wedge d \log b_{n-1} \in \Omega_{k/\bbZ} ^{n-1}$, observe that
\begin{eqnarray*}
d (a d \log b_1 \wedge \cdots \wedge b_{n-1} ) &=& da \wedge d \log b_1 \wedge \cdots \wedge d \log b_{n-1} \\
&=& a d \log a \wedge d \log b_1 \wedge \cdots \wedge d \log b_{n-1} \\
&=& \phi (a \otimes a \wedge b_1 \wedge \cdots \wedge b_{n-1}).
\end{eqnarray*}

}
\end{remark}

This simple observation combined with Lemma \ref{differential} implies the following result.

\begin{corollary}\label{differential modulo exact}The map $\phi$ induces an isomorphism of $k$-vector spaces
$$\phi: \frac{ k \otimes \wedge^n k^{\times}}{\calR '} \overset{\sim}{\to} \Omega_{k/\bbZ} ^n / d \Omega_{k/\bbZ} ^{n-1},$$ where $\calR'$ is the vector subspace spanned by all elements of type
$$ a \otimes a \wedge b_1 \wedge \cdots \wedge b_{n-1},$$
for $a \in k, b_i \in k^{\times}$.
\end{corollary}

\begin{remark}\label{motivation 2}{\rm From the proof of Theorem 6.4 in \cite{BE2} that $ACH_0 (k, n) \simeq \Omega_{k/\bbZ} ^n$, we know that the generators of $\calR'$ correspond to the closed points
$$ \left( \frac{1}{a}, a, b_1, \cdots, b_{n-1} \right) \in \bbA^1 \times \square^{n-1}.$$
}
\end{remark}

Our motivic operator $\delta$ in the next section is motivated by this point.

\subsection{The operator $\delta$}

Let $A = k[x] / (x^m)$ for some $m \geq 2$. For integers $1 \leq k \leq n+1$, define rational maps
$$\delta _k : \bbA^1 \times \square^n \cdots \to \bbA^1 \times \square^{n +1}$$
$$ (x, t_1, \cdots, t_n) \mapsto \left( x, t_1, \cdots, t_{k-1}, \frac{1}{x}, t_k, \cdots, t_n \right).$$ These rational maps naturally induce homomorphisms
$$\delta_k : \calZ_p (X \times \Diamond_n(A) ) \to \calZ_p (X \times \Diamond_{n+1}(A)).$$

\begin{lemma}\label{delta identity} For $i, j \in \{ 1, \cdots, n+1 \}$, we have
$$\tuborg \delta_i \delta_j = \delta_{j+1} \delta_i & \mbox{ if } i \leq j, \\ \delta _i \delta _j = \delta_j \delta_{i-1} & \mbox{ if } i > j. \sluttuborg$$
\end{lemma}

\begin{proof}Straightforward.
\end{proof}

Define the operator 
\begin{equation}\label{delta} \delta:= \sum_{k=1} ^{n+1} (-1)^k \delta_k : \calZ_p (X \times \Diamond_n(A)) \to \calZ_p (X \times \Diamond_{n+1}(A)).
\end{equation}

\begin{corollary} $\delta^2 =0$. Thus, $\left( \calZ_* (X \times \Diamond_*) (A), \delta \right)$ is a complex.
\end{corollary}
\begin{proof}It immediately follows from Lemma \ref{delta identity}.
\end{proof}

We want to know how $\delta_k$ interacts with the face maps $\partial_i ^j$.

\begin{lemma}\label{partial delta}We have the following identities:
\begin{equation}\label{partial first}
\tuborg \partial_i ^j \delta_k = \delta_{k-1} \partial_i ^j & \mbox{ if } i < k,\\
\partial_i ^j \delta_k = 0 & \mbox{ if } i = k, \\
\partial_i ^j \delta_k = \delta_k \partial_{i-1} ^j & \mbox{ if } i > k.\sluttuborg
\end{equation}
Equivalently,
\begin{equation}\label{delta first}
\tuborg \delta_k \partial_i ^j = \partial_{i+1} ^j \delta_k & \mbox{ if } k \leq i, \\
\delta_k \partial _i ^j = \partial _i \delta_{k+1} ^j & \mbox{ if } k \geq i.\sluttuborg
\end{equation} In particular, $\partial_{i+1} \delta_i = \partial_i \delta_{i+1}$.
\end{lemma}

\begin{proof}Tedious but straightforward calculations show them easily.
\end{proof}

Unfortunately these identities show that $\partial$ and $\delta$ do not interact very nicely. For instance, the reader can easily see that
\begin{equation}
\tuborg \partial \delta_1 = - \delta_1 \partial, \\ \partial \delta_{\rm last} = \delta_{\rm last} \partial, \sluttuborg
\end{equation} but for intermediate $\delta_k$, we do not really have very enlightening interactions. To remedy this situation, we replace the additive higher Chow complex $\left( \calZ_* (X \times \Diamond_* (A)), \partial \right)$ by a quasi-isomorphic subcomplex. We use the ideas given in Lemma 4.2.3 and Theorem 4.4.2 in \cite{Bl2}:

\begin{definition} Define $$\calZ_p (X \times \Diamond_n (A))_0: = \left( \bigcap _{i=1} ^{n-1} \ker ( \partial _i ^0 ) \cap \ker (\partial_i ^{\infty}) \right) \cap \ker (\partial_n ^{\infty}).$$
\end{definition} Note that $\partial_n ^0 = \partial_{\rm last} ^0$ is the only nontrivial face map. Let's denote this map by $\partial'$. Certainly ${\partial'}^2 = 0$. Then,

\begin{lemma}\label{homotopy} The inclusion $\left( \calZ_* (X \times \Diamond_* (A))_0, \partial' \right) \subset \left( \calZ_* (X \times \Diamond_*(A)), \partial \right)$ is a homotopy equivalence.
\end{lemma}

\begin{proof}The same arguments as in Lemma 4.2.3, Theorem 4.4.2 in \cite{Bl2} show this lemma.
\end{proof}

Let's call $\left( \calZ_* (X \times \Diamond_*(A))_0, \partial' \right)$ \emph{the reduced additive higher Chow complex} associated to $A$.

\begin{lemma}\label{descension}$\delta_k \left( \calZ_p (X \times \Diamond_n(A))_0 \right) \subset \calZ_p (X \times \Diamond_{n+1} (A)) _0$ for all $1 \leq k \leq n+1$.
\end{lemma}

\begin{proof}Let $x \in \calZ_p (X \times \Diamond_n (A))_0$ so that 
\begin{equation}\label{given identity} \partial_i ^j (x) = 0 \mbox{ for } \tuborg 1 \leq i \leq n-1, \\ i = n, j = \infty.\sluttuborg\end{equation}

If $i= n+1$ and $j = \infty$, then we always have $i \geq k$ so that $$\partial_{n+1} ^{\infty} \delta_k (x) = \tuborg 0 & \mbox{ if } i=k, \\ \delta_k \partial_n ^{\infty} (x) = 0 & \mbox{ if } i >k, \sluttuborg$$ by \eqref{partial first} and \eqref{given identity}.

If $1 \leq i \leq n$, then both $\partial_i ^j (x)$ and $\partial_{i-1}^j (x)$ are zero. Thus, by \eqref{partial first} $$\partial_i ^j \delta_k (x) = \tuborg\delta_{k-1} \partial_i ^j (x) = 0 & \mbox{ if } i < k, \\ 0 & \mbox{ if } i = k, \\ \delta_k \partial_{i-1} ^j (x) = 0 & \mbox{ if } i > k.\sluttuborg$$ 
This finishes the proof.
\end{proof}

\begin{corollary} The operator $\delta$ descends to the subgroups:
$$\delta: \calZ_p (X \times \Diamond_n (A))_0 \to \calZ_p (X \times \Diamond_{n+1}(A))_0.$$ 
\end{corollary}

On the level of these subgroups, the boundary $\partial'$ interacts very nicely with $\delta$:

\begin{theorem}\label{commutative} The operator $\delta$ and $\partial' = \partial_{\rm last} ^0$ satisfy $$\delta \partial' = \partial ' \delta.$$ 
In particular, the reduced additive higher Chow complexes form a bicomplex with two boundary maps $(\partial', \delta)$:
\begin{equation}\label{bicomplex}\xymatrix{  & \vdots \ar[d] ^{\partial'} & \vdots \ar[d] ^{\partial' } &  \\
\cdots  & \calZ_{p+1} (n+2)_0 \ar[d] ^{\partial'} \ar[l]_{\delta \ \ \ \ } & \calZ_{p+1} (n+1)_0 \ar[d] ^{\partial'}  \ar[l] _{\delta} & \cdots \ar[l]_{\ \ \ \ \ \ \delta} \\
\cdots & \calZ_p (n+1)_0 \ar[d] ^{\partial'} \ar[l]_{\delta \ \ \ \ } & \calZ_p (n )_0 \ar[d] ^{\partial'} \ar[l]_{\delta} & \cdots \ar[l]_{\ \ \ \ \ \ \delta} \\
& \vdots & \vdots &}\end{equation}
where $\calZ_p (n) _0:= \calZ_p (X \times \Diamond_n(A))_0$.
\end{theorem}

\begin{proof}It follows from the \eqref{partial first} and \eqref{delta first}. Indeed, for any $x \in \calZ_{r+1} (n+1)$,
\begin{eqnarray*}
\delta\partial' (x) & =&  \delta \partial_{n+1} ^0 (x) \\
& = & \sum_{k=1} ^{n+1} (-1)^k \delta_k \partial_{n+1} ^0 (x) \\
& = & \sum_{k=1} ^{n+1} (-1)^k \partial _{n+2} ^0 \delta_k (x) \\
&= & \partial_{n+2} ^0 \left( \sum_{k=1} ^{n+1} (-1)^k \delta_k \right) (x) \\
& = & \partial_{n+2} ^0 ( \delta - (-1)^{n+2} \delta_{n+2}) (x).
\end{eqnarray*}
But by \eqref{partial first} we have $\partial_{n+2} ^0 \delta _{n+2} = 0$. Thus, the last expression is $\partial' \delta (x)$, as desired.\end{proof}

\begin{corollary}The $\delta$ induces a map
$$\delta_* : ACH_p (X ,n ; A ) \to ACH_p (X, n+1; A).$$ In particular, when $X = \spec (k)$, $A = k[x] / (x^2)$, and $p=0$, the map
$$\delta_*: \Omega_{k/\bbZ} ^n \to \Omega_{k/\bbZ} ^{n+1}$$ is identical to $ (n+1)d$, where $d$ is the exterior derivation.
\end{corollary}
\begin{proof}The first assertion is obvious. The second assertion follows from Theorem 6.4 of \cite{BE2}, Lemma \ref{differential}, Remark \ref{motivation 1}, and Remark \ref{motivation 2}.
\end{proof}
Recall that the maps $\delta_k$ came from rational maps between varieties. Thus, we proved that
\begin{corollary}The exterior derivation is motivic.
\end{corollary}

\section{Cyclic higher Chow theory}

In this section we propose a cycle complex that behaves like the cyclic homology out of the bicomplex $\calZ_r (n)$ of \eqref{bicomplex}. Note that $\partial'$ decreases both $p$ and $n$ by $1$, whereas $\delta$ increases $n$ by $1$ but does not change $p$. If we let $\calZ(n) := \bigoplus_p \calZ_p(n)_0$, then the bicomplex $\calZ(n)$ looks as follows:
\begin{equation}
\xymatrix{\vdots \ar[d] _{\partial'} & \vdots \ar[d] _{\partial'} & \vdots \ar[d]_{\partial'} & \vdots \ar[d]_{\partial'}  \\
\calZ (3) \ar[d] _{\partial'} & \calZ(2) \ar[l] _{\delta} \ar[d] _{\partial'} & \calZ (1) \ar[l]_{\delta} \ar[d] _{\partial'} & \calZ (0) \ar[l]_{\delta} \\
\calZ (2) \ar[d] _{\partial'} & \calZ (1) \ar[l]_{\delta} \ar[d]_{\partial'} & \calZ (0) \ar[l]_{\delta} & \\
\calZ(1) \ar[d] _{\partial'} & \calZ (0)\ar[l]_{\delta} & & \\
\calZ(0) & & &}
\end{equation}

Let $\calB \calZ$ be this bicomplex. A perceptive reader will notice that this is actually a mixed complex in the sense of A. Connes (\emph{c.f.} p. 79 in \cite{L}). We apply the usual formalism of mixed complexes to $\calB \calZ$. By definition, its homology $H_n(\calB \calZ)$ is the homology of the first column, \emph{i.e.} the additive higher Chow group $ACH_* (X, n; A)$. Its \emph{cyclic homology} $HC_n(\calB \calZ)$ is the homology $H_n(Tot (\calB \calZ))$ of the total complex. Notice that the bicomplex \eqref{bicomplex} itself is not a mixed complex, but since $\calB \calZ$ is the direct sum of these, the groups $H_n(\calB \calZ)$ and $HC_n(\calB  \calZ)$ have natural decompositions.

\begin{definition}The \emph{cyclic} (or, \emph{additive commutative}) Chow group $CCH_* (X, n;A)$ is the cyclic homology $HC_n (\calB \calZ)$ of the bicomplex $\calB \calZ$. The group $CCH_p (X, n;A)$ is the direct summand of $CCH_* (X,n;A)$ that comes from the diagonal of \eqref{bicomplex} that contains $\calZ _p(n)$ in the first column.
\end{definition}

\begin{remark}{\rm Notice that, despite the indices $(p, n)$ of the group $CCH_p (X, n;A)$, by definition this group contains cycles not just from $\calZ_p (n)_0$, but from
$$\bigoplus _{i \geq 0} ^{\min\left\{ p, \left\lfloor \frac{n}{2}\right\rfloor \right\}}\calZ_{p-i} (n-2i)_0.$$
}
\end{remark}

\begin{remark}{\rm If we work with the additive higher Chow complex, not the reduced one, then we have serious difficulties due to the absence of the commutativity of $\delta$ and $\partial$.
}
\end{remark}

Following the formalism of mixed complexes, we have the long exact sequence of complexes
$$0 \to (\calZ(*), \partial') \overset{I}{\to} Tot (\calB \calZ ) \overset{S}{\to} Tot \left( \calB \calZ [1,1] \right) \to 0.$$ Notice that $Tot \left( \calB \calZ [1,1] \right) = \left(Tot (\calB \calZ )\right) [2]$. Thus, we obtain the homology long exact sequence, which is the Connes periodicity exact sequence, that decomposes as follows:

\begin{theorem}\label{Connes periodicity}We have the Connes periodicity exact sequence involving $ACH$ and $CCH$:
$$\cdots \overset{B}{\to} ACH_p (n) \overset{I}{\to} CCH_p (n) \overset{S}{\to} CCH_{p-1} (n-2) \overset{B}{\to} ACH_{p-1} (n-1) \overset{I}{\to} \cdots,$$ where $ACH_p (n) := ACH_p (X, n; A)$ and $CCH_p (n) := CCH_p (X, n;A)$. The maps $I, S, B$ have bidegrees $(0,0), (-1, -2), (0, +1)$ in $(p,n)$, respectively.\end{theorem}

Using the functoriality of the additive higher Chow complex for projective morphisms $f: X \to Y$ and for flat morphisms $f: X \to Y$ of two varieties $X, Y$ of finite type over $k$ (Lemma 3.6 and Lemma 3.7 of Krishna and Levine \cite{KL}), up to a possible shift of indices, we can see that the above Connes periodicity sequence is functorial for a projective morphism and a flat morphism, up to a possible shift of degrees.

Now we can provide the missing C4) of the section \S 1:

\begin{theorem}\label{CCH_0} We have an isomorphism $$CCH_0 (k,n):= CCH_0 (k, n ; k[x] / (x^2))\simeq \Omega_{k/\bbZ}^n / d \Omega_{k/\bbZ} ^{n-1}.$$
\end{theorem}

\begin{proof}By definition, $$CCH_0 (k, n) = \frac{\calZ_0 (k, n)_0}{ \partial' \calZ_1 (k, n+1)_0 + \delta \calZ_0 (k, n-1)_0}.$$ By Theorem 6.4 in \cite{BE2}, we have $\calZ_0 (k,n)_0 / \partial' \calZ_1 (k, n+1)_0 \simeq \Omega_{k/\bbZ} ^n$, whereas by Remark \ref{motivation 1}, Corollary \ref{differential modulo exact}, and Remark \ref{motivation 2}, elements of the group $\delta \calZ_0 (k, n-1)_0$ are exact forms. This finishes the proof.
\end{proof}

\section{The shuffle product and the wedge product} In this section we define the shuffle product structure on the classes of additive higher Chow groups. Together with the multiplication of the algebraic group $\mathbb{G}_m$, we define the wedge product for the additive higher Chow groups with modulus. 

\subsection{Permutations} \subsubsection{Definitions} We use the following notations:
\begin{enumerate}
\item [$\bullet$] For integers $r \geq 0$, let $\perm_{r}$ be the group of permutations on $\{1, 2, \cdots, r \}$. 
\item [$\bullet$] For an integer $s \geq 1$ and integers $p_1, p_2, \cdots, p_s \geq 0$, a $(p_1, \cdots, p_s)$-shuffle is a permutation $\sigma  \in \perm_{p_1 + \cdots + p_s}$ such that
$$\tuborg  \sigma (1) < \sigma (2) <  \cdots  < \sigma (p_1),  \\
\vdots  \\
 \sigma (p_1 + \cdots + p_i +1 ) < \sigma (p_i + 2) < \cdots  < \sigma (p_1 + \cdots + p_i + p_{i+1}), \\
\vdots  \\
 \sigma (p_1 + \cdots + p_{s-1} + 1) < \sigma (p_1 + \cdots + p_{s-1} + 2) < \cdots  < \sigma (p_1 + \cdots + p_{s-1} + p_s). \sluttuborg$$
We denote by $\perm_{(p_1, \cdots, p_s)}$ the set of all $(p_1, \cdots, p_s)$-shuffles in $\perm_{p_1 + \cdots + p_s}$. 
\end{enumerate} 

Note that $|\perm_{(p_1, \cdots, p_s)} | = \frac{ (p_1 + \cdots + p_s)!}{p_1 ! \cdots p_s !}$. Also, $\perm _{r} = \perm _{\underset{r}{(\underbrace{1, \cdots, 1}})}$

We will use the double shuffles $\perm_{(r,s)}$ and triple shuffles $\perm_{(1, r, s)}$ in this paper.

\subsubsection{Permutation actions}

Let $A \in (Art/k)$ with $\edim (A) = e \geq 1$. Consider the associated space $\Diamond_r (A)= \mathbb{A}^e \times \square^r$. For $k$-rational closed points, a permutation $\sigma \in \perm_r$ acts naturally via 
$$\sigma \cdot ( x, t_1, \cdots, t_r) := \left( x, t_{\sigma^{-1} (1)}, \cdots, t_{\sigma^{-1} (r)} \right).$$ This action naturally generalizes to all closed subvarieties of $\Diamond_r (A)$ as well. Furthermore, it sends admissible cycles to admissible cycles in $\Diamond_r (A)$. 

\subsubsection{Some properties of multiple shuffles}The following lemmas on double shuffles and triple shuffles will play important roles in the Proposition \ref{Connes is a derivation}. It is not necessary to read them now.

\begin{lemma}Permutations $\tau$ in $\perm_{(1, n)}$ are in one to one correspondence with the numbers $\{1 , \cdots, n \}$, where the correspondence is given by $\tau \leftrightarrow \tau (1)$.
\end{lemma}
\begin{proof}
Obvious.\end{proof}

\begin{definition}For permutations $\sigma \in \perm _{n}$ and $\tau \in \perm_{(1, n)}$ with $\tau(1) = i \in \{ 1, \cdots, n \}$, define the permutation $ \sigma_{\tau} =\sigma[i] \in \perm_{n+1}$ by sending

$$j \in \{ 1, \cdots, n+1\} \mapsto \tuborg \sigma (j) & \mbox{ if } j < i, \\ j & \mbox{ if } j = i, \\ \sigma (j-1) & \mbox{ if } j > i.\sluttuborg$$
\end{definition}

\begin{lemma}\label{0-1} Let $\sigma \in \perm_{(r,s)}$ and $\tau \in \perm_{(1, r+s)}$. Then, the product $\sigma_{\tau} \cdot \tau$ in $\perm_{r+s+1}$ is a $(1, r, s)$-shuffle, \emph{i.e.} $\sigma_{\tau} \cdot \tau \in \perm _{(1, r, s)}$. Furthermore, the set-theoretic map
$$\phi_1: \perm_{(r,s)} \times \perm_{(1, r+s)} \to \perm _{(1, r, s)}$$
$$\left( \sigma , \tau \right) \mapsto \sigma_{\tau} \cdot \tau$$ is a bijection.
\end{lemma}

\begin{proof}The first statement is obvious. For the second statement, the surjectivity part is obvious by keeping track of where $1$ is sent. But since both sides have the cardinality $\frac{ (r+s)!}{ r! s!} \frac{ (r+s + 1)!} {(r+s)!} = \frac{ (r+s+1)!}{r! s!}$, the map $\phi_1$ must be bijective.\end{proof}

\begin{lemma}\label{multiple shuffle 1}In the group ring $\mathbb{Z} [ \perm_{r+s+1}]$, we have
$$\sum_{\sigma \in \perm_{(r,s)} } (\sgn (\sigma)) \left( \sum_{\tau \in \perm _{(1, r+s)}} (\sgn (\tau)) \sigma_{\tau} \cdot \tau \right) = \sum_{\nu \in \perm_{(1, r, s)}} (\sgn (\nu)) \nu.$$
\end{lemma}

\begin{proof}Note that $\sgn (\sigma_{\tau} \cdot \tau) = \sgn (\sigma) \sgn (\tau)$. Thus, together with the Lemma \ref{0-1}, we get the desired result.\end{proof}

\begin{lemma}\label{2-0}For $\sigma \in \perm_{(r+1, s)}$ and $\tau \in \perm_{(1, r)}$, the permutation $\sigma \cdot \left( \tau \times {\rm Id}_s \right)$ is in $\perm_{(1, r, s)}$. Furthermore, the set-theoretic map
$$\phi_2 : \perm _{(r+1, s)} \times \perm _{(1, r)} \to \perm _{(1, r, s)}$$
$$ (\sigma, \tau) \mapsto \sigma \cdot \left( \tau \times {\rm Id}_s \right)$$ is a bijection.\end{lemma}

\begin{proof}The first statement is obvious. For the second statement, the surjectivity part is obvious by keeping track of where $1$ is sent. But since both sides have the cardinality $\frac{ (r+s+1)!}{(r+1)! s!} \frac{ (r+1)!}{ r!} = \frac{ (r+s+1)!}{ r! s!}$, the map $\phi_2$ must be bijective.\end{proof}

\begin{lemma}\label{multiple shuffle 2}In the group ring $\mathbb{Z} [ \perm_{r+s+1}]$, we have
$$ \left( \sum_{\sigma \in \perm_{(r+1, s)}} (\sgn (\sigma)) \sigma \right) \left( \sum_{\tau \in \perm_{(1, r)}} (\sgn (\tau)) (\tau \times {\rm Id} _s ) \right) = \sum_{\nu \in \perm_{(1, r,s)}} (\sgn (\nu)) \nu.$$\end{lemma}

\begin{proof}It follows immediately from Lemma \ref{2-0} together with the observation that $(\sgn (\sigma)) (\sgn (\tau)) = (\sgn (\sigma \cdot (\tau \times {\rm Id}_s))).$\end{proof}

\begin{lemma}\label{3-0}For $\sigma \in \perm_{(r, s+1)}$ and $\tau \in \perm _{(1, r+s)}$, the permutation $\sigma \cdot ({\rm Id}_r \times \tau)$ is in $\perm _{(1, r,s)}$. Furthermore, the set-theoretic map
$$\phi_3: \perm_{(r, s+1)} \times \perm _{(1, r+s)} \to \perm _{(1, r, s)} $$
$$ (\sigma, \tau ) \mapsto \sigma \cdot ( {\rm Id}_r, \times \tau)$$ is a bijection.\end{lemma}

\begin{proof}The proof is similar to Lemma \ref{2-0}.\end{proof}

\begin{lemma}\label{multiple shuffle 3}In the group ring $\mathbb{Z} [ \perm_{ r + s + 1}]$, we have
$$(-1)^r \left( \sum_{\sigma \in \perm _{(r, s+1)}} (\sgn (\sigma)) \sigma \right) \left( \sum_{\tau \in \perm_{(1, r+s)}} (\sgn (\tau)) ({\rm Id}_r \times \tau) \right)$$ $$ = \sum_{\nu \in \perm _{(1, r,s)}} (\sgn (\nu)) \nu.$$\end{lemma}

\begin{proof}It follows immediately from Lemma \ref{3-0} together with the observation that $(-1)^r (\sgn (\sigma)) (\sgn ({\rm Id}_r \times \tau)) = (\sgn (\sigma \cdot ({\rm Id}_r \times \tau))).$\end{proof}

\subsection{The shuffle product} Although we can develop our theory more generally, for simplicity of notations we work with $X = \spec (k)$. Let $\left( A_1, \mathfrak{m}_1 \right), \left( A_2, \mathfrak{m}_2 \right) \in \left(Art/k \right)$ be two Artin local $k$-algebras with $\edim (A_i ) = e_i \geq 1$. We will always identify the product $\Diamond_{r_1} (A_1) \times \Diamond _{r_2} (A_2) = \mathbb{A}^{e_1} \times \square^{r_1} \times \mathbb{A} ^{e_2} \times \square^{r_2}$ with $\mathbb{A}^{e_1+ e_2} \times \square^{r_1 + r_2}$.

\begin{lemma}Let $Z_1 \subset \Diamond_{r_1} (A_1), Z_2 \subset \Diamond_{r_2} (A_2)$ be admissible irreducible closed subvarieties. Then, their product $Z_1 \times Z_2 \subset \mathbb{A}^{e_1 + e_2} \times \square^{r_1 + r_2}$ is also an admissible closed subvariety in $\Diamond_{r_1 + r_2} (A_1 \otimes_k A_2)$.
\end{lemma}
\begin{proof}Out of the requirements for admissibility, only the modulus condition is less trivial. Fix presentations
$$A_1 \simeq k[ \mathbb{X}_1]/ J_1 , ~~ A_2 \simeq k[\mathbb{X}_2]/J_2$$ as in \S 2. The ring $\left( A_1 \otimes_k A_2 , \left< \mathfrak{m}_1 , \mathfrak{m}_2 \right> \right)$ is in $(Art/k)$, and 
$$A_1 \otimes _k A_2 \simeq k[ \mathbb{X}_1 \cup \mathbb{X}_2]/ \left< J_1 , J_2 \right>.$$ Note that the embedding dimension of $A_1 \otimes _k A_2$ is $e_1 + e_2$.

Let $\widehat{Z_i} \subset \mathbb{A}^{e_i} \times \left( \mathbb{P} \right)^{r_i}$ be the Zariski closures of $Z_i$, where $i = 1, 2$. Let $\nu_i : \overline{Z_i} \to \widehat{Z_i}$ be their normalizations. Recall (\emph{c.f.} Lemma 3.1 in \cite{KL}) that the product of two reduced normal finite type $k$-schemes is again normal over perfect fields. Thus, 
$$\nu = \nu_1 \times \nu_2 : \overline{Z_1} \times \overline{Z_2} \to \widehat{Z_1} \times \widehat{Z_2} = \widehat{Z_1 \times Z_2} \subset \mathbb{A}^{e_1 + e_2} \times \square^{r_1 + r_2}$$ is a normalization of $\widehat{Z_1 \times Z_2}$. We will identify $ \overline{Z_1} \times \overline{Z_2}$ with $\overline{Z_1 \times Z_2}$.

For a closed point $\mathfrak{p} \in \supp \left( \nu^* V (\left< \mathfrak{m}_1, \mathfrak{m}_2 \right>)\right)$ in $\overline{Z_1} \times \overline{Z_2}$, there correspond closed points $\mathfrak{p}_i \in \supp \left( \nu_1 ^* V (\mathfrak{m}_i) \right)$ in $\overline{Z_i}$ for $i=1, 2$. Thus by the modulus conditions for $Z_1$ and $Z_2$, there exist indices $j_1 \in \left\{ 1, \cdots, r_1\right\}$ and $j_2 \in \left\{r_1 + 1 , \cdots, r_1 + r_2 \right\}$ such that
$$ 1- t_{j_1} \in (J_1) \cdot \mathcal{O}_{\overline{Z_1}, \mathfrak{p}_1}, ~~ 1- t_{j_2} \in (J_2) \cdot \mathcal{O} _{\overline{Z_2}, \mathfrak{p}_2}.$$ Via the natural maps
$$(J_i) \cdot \mathcal{O}_{\overline{Z}_i, \mathfrak{p}_i} \to (J_i) \cdot \mathcal{O}_{\overline{Z_1 \times Z_2}, \mathfrak{p}} \to \left< J_1, J_2 \right> \cdot \mathcal{O}_{\overline{Z_1 \times Z_2}, \mathfrak{p}} \ \mbox{ for } i =1, 2,$$ we have the both $1-t_{j_1}, 1-t_{j_2} \in \left< J_1, J_2 \right> \cdot \mathcal{O}_{\overline{Z_1 \times Z_2} , \mathfrak{p}}.$ This proves the modulus condition.\end{proof}

\begin{definition} Extend the above product $\mathbb{Z}$-bilinearly to obtain the \emph{concatenation product}
$$\times = \times_{(r_1, r_2)} : \mathcal{Z}_p (\Diamond_{r_1} (A_1)) \otimes \mathcal{Z}_q (\Diamond_{r_2} (A_2)) \to \mathcal{Z}_{p+q} \left( \Diamond_{r_1 + r_2} (A_1 \otimes _k A_2) \right).$$
\end{definition}

\begin{lemma}\label{partial is a derivation for concatenation}For $x \in \calZ_p (\Diamond_{r_1} (A_1))$ and $y \in \calZ_q (\Diamond_{r_2} (A_2))$, we have $$\partial (x \times y) = (\partial x) \times y + (-1)^{r_1} x \times (\partial y).$$
\end{lemma}
\begin{proof}This is obvious from the definition of $\partial$.\end{proof}

\begin{definition}Let $d \geq 0$ be an integer. An admissible irreducible closed subvariety $Z \in \mathcal{Z}_d (\Diamond_n (A_1 \otimes
A_2))$ is said to be \emph{$(A_1, A_2)$-decomposable} if for some permutation $\sigma \in \perm_n$ and integers $r_1,
r_2 \geq 1$ such that $r_1 + r_2 = n$, the variety $ \sigma \cdot Z $ is in the image of the concatenation
product $\times_{(r_1, r_2)}$. The subgroup generated by $(A_1, A_2)$-decomposable cycles will be denoted by
$\mathcal{Z}_d (\Diamond_n (A_1 \otimes _k A_2)) ^{dec(A_1, A_2)}$, or just $\mathcal{Z}_d (\Diamond_n (A_1
\otimes _k A_2))^{dec}$ if reference to $(A_1, A_2)$ is unnecessary.
\end{definition}

\begin{definition} Let $p, q \geq 0$ be integers. Let $Z_i \subset \Diamond_{r_i} (A_i)$, $i=1,
2$, be admissible irreducible closed subvarieties of dimension $p$ and $q$, respectively. Let $d= p+q$ and $n = r_1 + r_2$. Define the \emph{$(r_1, r_2)$-shuffle product} $\times_{sh(r_1,
r_2)}$ by
$$Z_1 \times_{sh(r_1, r_2)} Z_2 := \sum_{\sigma \in \perm_{(r_1, r_2)}} \sgn(\sigma) \sigma \cdot \left( Z_1
\times Z_2 \right) \in \calZ_d(\Diamond_{n} (A_1 \otimes_k A_2)).$$ When the reference to $(r_1, r_2)$ is
unnecessary, we may write $\times_{sh}$ instead of $\times_{sh(r_1, r_2)}$. By extending it
$\mathbb{Z}$-bilinearly, we have the $(r_1, r_2)$-shuffle product
$$\times_{sh} = \times_{sh(r_1, r_2)} : \mathcal{Z}_p (\Diamond_{r_1} (A_1)) \otimes \mathcal{Z}_q (\Diamond
_{r_2} (A_2)) \to \mathcal{Z}_{d} (\Diamond_{n} (A_1 \otimes _k A_2)).$$
\end{definition}

\begin{remark}\label{commutative graded}Note that in general we do not have $x \times y = (-1)^{r_1r_2} y \times x$, but we do have $x \times _{sh} y = (-1)^{r_1 r_2} y \times_{sh} x$. The proof just follows from the definition of the double shuffles and the shuffle product.\end{remark}

\begin{proposition}\label{partial derivation for shuffle}For the intersection boundary map $\partial
: \mathcal{Z}_{d}( \Diamond_{n} (A_1 \otimes _k A_2)) \to \mathcal{Z}_{d-1} (\Diamond_{n-1}
(A_1 \otimes _k A_2))$ and two admissible cycles $x \in \mathcal{Z}_p (\Diamond_{r_1} (A_1))$ and $y \in
\mathcal{Z}_{q} (\Diamond_{r_2} (A_2))$ with $p + q = d$ and $r_1 + r_2 = n$, we have
$$\partial (x \times_{sh} y ) = ( \partial x) \times_{sh} y + (-1)^{r_1} x \times_{sh} (\partial y).$$ In other
words, the boundary map $\partial$ is a graded derivation for the shuffle product $\times_{sh}$.
\end{proposition}
\begin{proof} It follows from Lemma \ref{partial is a derivation for concatenation}. This is just a matter of rewriting the definition of $\partial$ and $\times_{sh}$ carefully
midning signs.(\emph{c.f.} Proposition 4.2.2. on p. 123 in \cite{L})\end{proof}

Generally for a fixed integer $d \geq 0$ and $n \geq 0$, we can define the \emph{total shuffle product} as the
sum of all possible shuffle products:
\begin{eqnarray*}\times_{sh} := \sum_{r_1 + r_2 = n} \times _{sh (r_1, r_2)} : 
\bigoplus_{r_1 + r_2 = n} \bigoplus _{p + q = d} \mathcal{Z}_p (\Diamond_{r_1} (A_1)) \otimes \mathcal{Z}_{q}
(\Diamond_{r_2} (A_2)) \\ \to \mathcal{Z}_d (\Diamond_{n} (A_1 \otimes _k A_2)).\end{eqnarray*}

For the tensor product of two total additive Chow complexes $\mathcal{Z} (\Diamond_* (A_i)) = \bigoplus_{p \geq 0}
\mathcal{Z}_p (\Diamond_* (A_i))$, by the Proposition \ref{partial derivation for shuffle} we have a homomorphism of
complexes
\begin{eqnarray*}\times_{sh} : \left( \calZ(\Diamond_* (A_1)) \otimes \calZ(\Diamond_* (A_2)), \partial \otimes {\rm Id} + (-1)^*
{\rm Id} \otimes \partial \right) \to \left( \calZ (\Diamond_* (A_1 \otimes_k A_2 )), \partial \right).\end{eqnarray*}

\begin{corollary}Let $d = p+q$ and $n = r_1 + r_2$. The shuffle product $\times_{sh}$ induces homomorphisms
$$sh_* : ACH_p (k, r_1 ; A_1 ) \otimes ACH_q (k, r_2 ; A_2) \to ACH_{d} (k, n; A_1 \otimes _k A_2),$$
$$sh_*: ACH_* (k, * ; A_1 ) \otimes ACH_* (k, * ; A_2) \to ACH_* (k, *; A_1 \otimes_k A_2).$$

\end{corollary}

\begin{remark}One may wonder if this $sh_*$ is an isomorphism. Unfortunately to apply the argument of the proof of the
Eilenberg-Zilber theorem (\emph{c.f.} Theorem 4.2.5. in \cite{L} and Theorem 8.1. in \cite{ML}), one needs an analogue
of the deconcatenation operation on tensor coalgebras, which is an unlikely one for cycles because we do not in general
has a natural way of decomposing a closed subvariety $Z \subset \mathbb{A}^e \times \square^n$ into the concatenation of
two $Z_1 \subset \mathbb{A}^{e_1} \times \square^r$ and $Z_2 \subset \mathbb{A}^{e- e_1} \times \square^{n-r}$ even up to boundary. 
\end{remark}

\subsection{The wedge product}
Let $A_1, A_2 \in (Art/k)$ be Artin local $k$-algebras of $\edim (A_i) = 1$. Thus, $A_i \simeq k[x]/(x^{m_i})$ for some
$m_i \geq 1$.
The groups $ACH_p (X, n; A_i)$ are equal to the additive Chow groups with modulus $m_i$ in \cite{P1, R}. For an
indeterminate $w$, define a $k$-algebra homomorphism
$$\phi: k[w] \to A_1 \otimes_k A_2 \simeq k[x,y] / (x^m_1, y^{m_2}), w \mapsto xy.$$ Notice that $\ker \phi = (w^m)$
where $m = \min\{ m_1, m_2 \}$. 

\begin{definition}Define $\min\{A_1, A_2\}= \min\{k[x]/(x^{m_1}), k[x]/(x^{m_2}) \} := k[w]/(w^m)$. \end{definition}

Recall that the linear algebraic group $\mathbb{G}_m$ has the multiplication $\mu: \mathbb{G}_m \times \mathbb{G}_m \to
\mathbb{G}_m$. This extends to $\mu: \mathbb{A}^1 \times \mathbb{A}^1 \to \mathbb{A}^1$, and it gives a morphism of
varieties $\mu: \mathbb{A}^2 \times \square^n \to \mathbb{A}^1 \times \square^n$.

For a $k$-rational point $(x, y, t_1, \cdots, t_n) \in \mathbb{A}^2 \times \square^n = \Diamond_n (A_1 \otimes_k A_2)$,
this $\mu$ induces an action
$$\mu_* (x, y, t_1, \cdots, t_n) := (xy, t_1, \cdots, t_n) \in \mathbb{A}^1 \times \square^n.$$ For a general admissible
closed subvariety $Z \subset \mathbb{A}^2 \times \square^n = \Diamond_n (A_1 \otimes_k A_2)$, since $\mu$ is not a
closed morphism (but it is an open morphism), the set $\mu(Z)$ is not necessary closed. Thus, we define 
$$\mu_* (Z) := cl (\mu (Z)) \subset \mathbb{A}^1 \times \square^n$$ where $cl ( \cdot )$ means the Zariski closure in
the space. One difficulty is that $\mu_*$ does not always respect the admissibility conditions, especially the modulus
condition. 

\begin{lemma}\label{domain of product}Let $Z \subset \Diamond_n (A_1
\otimes_k A_2)$ be an admissible irreducible closed subvariety. Let $(x, y, t_1, \cdots, t_n) \in \mathbb{A}^2 \times
\square^n = \Diamond_n (A_1 \otimes_k A_2)$ be the coordinates. Let $A = \min \{ A_1, A_2 \}$.

 If $Z$ is a $(A_1, A_2)$-decomposable cycle, then $\mu_*(Z)$ is admissible in $\Diamond_n (A)$.

\end{lemma}

\begin{proof}We may assume that $Z$ is irreducible and $Z = Z_1 \times Z_2$ for some irreducible $Z_1 \in \calZ_p
(\Diamond_{r_1}(A_1))$ and $Z_2 \in \calZ_{q} (\Diamond_{r_2} (A_2))$, where $d = p+q$ and $n = r_1 + r_2$. Note that the Zariski closure of $cl (\mu(Z))$ in
$\mathbb{A}^1 \times \left( \mathbb{P}^1 \right)^n$ is equal to the Zariski closure of $\mu(Z)$ in the same space. Let
$\widehat{\mu(Z)}$ be the Zariski closure, and let $\nu: \overline{\mu(Z)} \to \widehat{\mu(Z)} \subset \mathbb{A}^1
\times \left( \mathbb{P}^1 \right)^n$ be its normalization. Consider the normalization $\nu' : \overline{Z} \to
\widehat{Z}$ of the Zariski closure of $Z$ in $\mathbb{A}^2 \times \left( \mathbb{P}^1 \right)^n$. By the universality of normalization $\nu$, we have a map $\overline{\mu} : \overline{Z} \to \overline{\mu(Z)}$ that fits into the commutative diagram:
$$\xymatrix{ \overline{Z} \ar[r] ^{\nu'} \ar[d]^{\overline{\mu}} & \widehat{Z} \ar[d] ^{\mu} \ar[r] &
\mathbb{A}^2 \times \left( \mathbb{P}^1 \right)^n \ar[d] ^{\mu} \\
\overline{\mu(Z)} \ar[r]^{\nu} & \widehat{\mu(Z)} \ar[r] & \mathbb{A}^1 \times \left( \mathbb{P}^1
\right)^n}$$

Let $(w, t_1, \cdots, t_n) \in \mathbb{A}^1 \times \left( \mathbb{P}^1 \right)^n$ be the coordinates.

Let $\mathfrak{p} \in \supp \left( \nu^* \{ w = 0 \} \right)$ be a closed point on $\overline{\mu(Z)}$. Then,
$\overline{\mu} ^{-1} (\mathfrak{p})$ lies in $\supp ({\nu'}^*\{ x=0\} + {\nu'}^* \{ y=0 \})$. Pick any point
$\mathfrak{q} \in \overline{\mu} ^{-1} (\mathfrak{p})$. Since $Z$ is $(A_1, A_2)$-decomposable, we have either
$$1- t_{j_1} \in (x^{m_1}) \cdot \mathcal{O}_{\overline{Z_1}, \pi_1 (\mathfrak{q})} \mbox{ for some } 1 \leq j_1 \leq
r_1,
$$ or
$$ 1- t_{j_2} \in (y^{m_2}) \cdot \mathcal{O}_{\overline{Z_2}, \pi_2 (\mathfrak{q})} \mbox{ for some } r_1 + 1 \leq j_2
\leq n,$$ where $\pi_i$ are the projections from $\Diamond_{n}(A_1 \otimes_k A_2)$ to $\Diamond_{r_i} (A_i)$
for $i=1, 2$. But, either case, we have maps
$$\xymatrix{ (x^{m_1})\cdot \mathcal{O}_{\overline{Z_1}, \pi_1 (\mathfrak{q})} \ar[dr] &  & \\
& (x^{m_1}, y^{m_2})\cdot \mathcal{O}_{\overline{Z}, \mathfrak{q}} \ar[r] & (w^{m}) \cdot \mathcal{O}_{\overline{\mu(Z)}, \mathfrak{p}} \\
(y^{m_2})\cdot \mathcal{O}_{\overline{Z_2}, \pi_2 (\mathfrak{q})} \ar[ru] & &}$$ where $m = \min\{ m_1, m_2 \}$, so that the image of $1-t_{j_1}$ for some $1 \leq j_1 \leq r_1$ or $1- t_{j_2}$ for some $r_1 + 1 \leq j_2 \leq n$ lies in $(w ^m)\cdot \mathcal{O}_{\overline{\mu(Z)}, \mathfrak{p}}$. This proves the modulus condition. Proper intersections with faces are obvious.\end{proof}

\begin{lemma}The product $\mu_* : \calZ_d (\Diamond_{n} (A_1 \otimes_k A_2)) ^{dec (A_1, A_2)} \to \calZ_d (\Diamond_n (A))$ is $\partial$-equivariant, where $A= \min\{ A_1, A_2 \}$. Thus, $\mu_*$ induces a map 
$$\mu_*: ACH_d (\Diamond_n (A_1 \otimes_k A_2))^{dec(A_1, A_2)} \to ACH_d (\Diamond_n (A)).$$
\end{lemma}
\begin{proof}Obvious.\end{proof}

\begin{proposition}Let $A \in (Art/k)$ be an Artin local $k$-algebra with $\edim (A) = 1$. Then the map $\wedge:= \mu_* \circ sh_*$ makes $ACH_* (k, *; A)$ a commutative graded algebra.
\end{proposition}
\begin{proof}We apply the above discussion with $A_1 = A_2 = A$ and consider the shuffle product
$$sh_*: ACH_* (k, *;A) \otimes ACH_* (k, *; A) \to ACH_* (k, *; A \otimes_k A).$$ The map $\mu_*$ is not necessarily defined on all of $ACH_* (k, *; A \otimes_k A)$, but by Lemma \ref{domain of product}, it is defined on the image of $sh_*$ so that $\wedge$ is well-defined: 
$$\wedge= \mu_* \circ sh_*: ACH_* (k, *;A) \otimes ACH_* (k, *; A) \to ACH_* (k, *; A),$$
where $\min\{A, A \} = A$. That this is commutative in graded sense follows from Remark \ref{commutative graded}.\end{proof}

\section{CDGA of additive Chow groups}

In \S 3, we defined the motivic Connes boundary operator $\delta$ on the reduced additive Chow complex that induces
$$\delta_* : ACH_p (k, n; A) \to ACH_p (k, n+1; A), \mbox{ with } \delta_* ^2 = 0,$$ when $\edim (A) = 1$. On the other hand, in \S 5 we proved that $\left( ACH_* (k, *; A), \wedge \right)$ is a commutative graded algebra. In this section, we show that $\delta_*$ is actually a derivation for $\wedge$ thus $\left( ACH_* (k, *; A), \wedge, \delta_* \right)$ is the CDGA (commutative differential graded algebra). This result generalizes R\"ulling's theorem \cite{R} that $ACH_0 (k, *; A)$ is a CDGA of generalized de Rham-Witt complex. 

The central result is the construction of the ``cyclic shuffle product" that gives Proposition \ref{Connes is a derivation}. This construction is motivated by Chapter 4 of \cite{L}, but unlike this reference we do not actually use what Loday calls the cyclic shuffle product, although we use this name for a psychological reason. In our construction we use triple shuffles in $\perm_{(1, r_1, r_2)}$. Recall the following result from Lemma 2.5 in \cite{P2}: (\emph{c.f.} Proposition 6.3 in \cite{BE2}. This is a variation of a cycle of B. Totaro in \cite{T}.)
\begin{lemma}\label{cycle C_2}For $a \in k, b_1, b_2 \in k^{\times}$, define a $1$-cycles in $\mathcal{Z}_1 (\Diamond_2 (k[x]/ (x^2)))$
$$C_2 ^{a, (b_1, b_2)}:= \tuborg \left\{ \left( \frac{1}{a} , t, \frac{ b_1 t - b_1 b_2 }{ t- b_1 b_2} \right) | t \in k \right\} &\mbox{ if } a \not = 0, \\ 0 & \mbox{ if } a = 0.\sluttuborg$$ Then,
$$\partial C_2 ^{ a, (b_1, b_2)} = \left( \frac{1}{a} , b_1 \right) + \left( \frac{1}{a}, b_2 \right) - \left( \frac{1}{a} , b_1 b_2 \right),$$ where the symbol $\left( \frac{1}{a} , b \right)$ is interpreted as $0$ if $a = 0$.
\end{lemma}

The cycle $C_2 ^{a, (b_1, b_2)}$ is in fact in $\calZ_1 (\Diamond_2 (k[x]/(x^m)))$ for all $m \geq 2$.

\begin{definition}Let $ A = k[x]/ (x^m)$ and let  $p, q, r_1, r_2 \geq 0$ be integers. Let $n = r_1 + r_2$ and $d = p+q$. For two admissible irreducible closed subvarieties $Z_1 \in \calZ_p (\Diamond_{r_1}(A_1))$ and $Z_2 \in \calZ _q (\Diamond_{r_2} (A_2))$, motivated by the above Lemma, we define the \emph{extra-degenerate concatenation} $Z_1 \times ' Z_2$ as follows:  consider a locally closed space $\mathcal{M} \subset \mathbb{G}_m \times \mathbb{G}_m \times \mathbb{A}^1 \times \square^2 $ defined by the collection of curves parametrized by $(x, y) \in \mathbb{G}_m \times \mathbb{G}_m$ $$\mathcal{M} = \left\{ \left( x, y \right) \times \left( xy, t, \frac{yt-1}{xyt -1}\right) | x, y \in k^{\times}, t \in k \right\}$$ where we have two natural projections
$$ \xymatrix{ & \mathcal{M} \ar[dl] _{\pi_{\alpha}} \ar[dr]^{\pi _{\beta}} & \\ \mathbb{G}_m \times \mathbb{G}_m & & \mathbb{A}^1 \times \square^2 }.$$ We have a cross-section of $\pi_{\alpha}$:
$$ C: \mathbb{G}_m \times \mathbb{G}_m \cdots \to \mathcal{M}$$
$$ (x,y) \mapsto \mbox{ the curve } \left\{ \left( xy, t, \frac{yt -1}{xyt -1} \right) | t \in k \right\}.$$

It induces $C': \mathbb{G}_m \times \mathbb{G}_m \times \square ^n \overset{C \times ({\rm Id}_{\square^n})}{\longrightarrow} \mathcal{M} \times \square^{n} \overset{\pi_{\beta} \times {\rm Id}_{\square^n}  }{\longrightarrow} \mathbb{A}^1 \times \square^{n+2}$ Define $Z_1 \times' Z_2:=$ the Zariski closure of $ C'(Z_1 \times Z_2) $ in $\mathbb{A}^1\times \square^{n+2}$. In general, we have $\dim (Z_1 \times ' Z_2) = \dim Z_1 + \dim Z_2 + 1$.
\end{definition}

\begin{lemma}For the above $Z_1$ and $Z_2$, the extra-degenerate concatenation $Z_1 \times ' Z_2$ is admissible in $\Diamond_{n+2} (A_1 \otimes_k A_2)$
\end{lemma}

\begin{proof}The proper intersection condition is obvious. The modulus condition follows from the conditions for $Z_1$ and $Z_2$. 
\end{proof}

\begin{remark}
On $k$-rational points $(x, t_1, \cdots, t_{r_1} ) \in Z_1$ and $(y, t_{r_1+1}, \cdots, t_{n}) \in Z_2$, we have
$$(x, t_1, \cdots, t_{r_1} ) \times ' (y, t_{r_1 +1}, \cdots, t_{n})= C_2 ^{\frac{1}{xy} , \left( \frac{1}{x}, \frac{1}{y} \right)} \times (t_1, \cdots, t_{r_1}, t_{r_1 + 1}, \cdots, t_{n} ).$$ This is a $1$-cycle in $\mathcal{Z}_1 (\Diamond_{n + 2} (A))$. 
\end{remark}

\begin{definition} \label{definition cyclic shuffle product}Under the above assumptions, define the \emph{cyclic shuffle product} $\wedge'$ by
$$Z_1 \wedge' Z_2 := \sum_{\nu \in \perm _{(1, r_1, r_2)}} (\sgn(\nu)) \nu \cdot \left( Z_1 \times' Z_2 \right) \in \mathcal{Z}_{d+1} (\Diamond_{n + 2} (A)),$$ where $\nu \in \perm_{(1, r_1, r_2)}$ acts on the set $\{ (1,2), 3, 4, \cdots, n+2 \}$ of $(n+1)$-objects treating $(1,2)$ as a single block. We extend it $\mathbb{Z}$-bilinearly.
\end{definition}

\begin{proposition}\label{Connes is a derivation}Let $A= k[x]/(x^m)$, and let $ p, q, r_1, r_2 \geq 0$ be integers. Let $d = p + q$ and $n = r_1 + r_2$. Let $\xi \in \calZ_p (\Diamond_{r_1} (A))_{0}$ and $\eta \in \calZ_q (\Diamond_{r_2} (A))_{0}$ such that $\partial' (\xi) =0$ and $\partial' (\eta ) = 0$. Then,
\begin{eqnarray*} \delta_*(\xi \wedge \eta ) - (\delta _* \xi ) \wedge \eta - (-1)^{r_1} \xi \wedge (\delta _* \eta) 
 =  - \partial (\xi \wedge' \eta) \ \ \mbox{ in } \calZ_d (\Diamond_{n+1} (A)).\end{eqnarray*}

\end{proposition}

\begin{remark} Before we prove the Proposition, observe that the map $\delta_*$ can be written as a sum over the set $\perm _{(1, n)}$ of double shuffle permutations. Indeed, for a $k$-rational point $(x, t_1, \cdots, t_n)$, we have
\begin{eqnarray*}
\delta_* (x, t_1, \cdots, t_n) &=& \sum_{i=1} ^{n+1} (-1)^i ( x, t_1, \cdots, t_{i-1} , \underset{i^{\rm th}}{\underbrace{\frac{1}{x}}}, t_i, \cdots, t_n) \\
&=& - \sum_{\tau \in \perm _{(1, n)}} (\sgn (\tau)) \tau \cdot \left( x, \frac{1}{x}, t_1, \cdots, t_n\right).
\end{eqnarray*}
But, it is not an element of $\mathbb{Z}[\perm_{n+1}]$.
\end{remark}

\begin{proof}[Proof of Proposition \ref{Connes is a derivation}]It is enough to check the identity for $k$-rational points. Let $\xi = (x, t_1, \cdots, t_{r_1}), \eta = (y, t_{r_1 +1}, \cdots, t_n)$. Thus, 
$$\tuborg  \xi \times \eta = (x, y, t_1, \cdots, t_n ), \\
\xi \times' \eta = \left\{ \left( xy, t, \frac{yt-1}{xyt-1} \right) t \in k \right\} \times (t_1, \cdots, t_n).\sluttuborg$$ Let's compute each term.

\begin{eqnarray*}
& & \delta_* (\xi \wedge \eta) \\
&=& \delta_* \mu_* (\xi \times_{sh} \eta) = \delta_*\left( \sum_{\sigma \in \perm_{(r_1, r_2)}} (\sgn (\sigma)) \sigma \cdot \mu_* \left( \xi \times \eta \right) \right) \\
&=& \delta_* \left( \sum_{\sigma \in \perm_{(r_1, r_2)}} (\sgn (\sigma)) \sigma \cdot (xy, t_1, \cdots, t_n) \right) \\
&=& \sum_{\sigma \in \perm_{(r_1, r_2)}} (\sgn (\sigma)) \delta_* (xy, t_{\sigma^{-1} (1)}, \cdots, t_{\sigma^{-1} (n))})\\
&=& - \sum_{\sigma \in \perm _{(r_1, r_2)}} (\sgn(\sigma)) \sum_{\tau \in \perm _{(1, n)}} (\sgn (\tau)) \tau \cdot (xy, \frac{1}{xy} , t_{\sigma ^{-1} (1)}, \cdots, t_{\sigma ^{-1} (n)} )\\
&=& - \sum_{\sigma \in \perm_{(r_1, r_2)}} (\sgn (\sigma)) \sum_{\tau \in \perm_{(1, n)}} (\sgn (\tau)) (\sigma_{\tau} \cdot \tau) \cdot (xy, \frac{1}{xy}, t_1, \cdots, t_n) \\
&=& - \left( \sum_{\nu \in \perm_{(1, r_1, r_2)}} (\sgn (\nu)) \nu\right)  \cdot (xy, \frac{1}{xy}, t_1, \cdots, t_n),
\end{eqnarray*}where the last equality follows from Lemma \ref{multiple shuffle 1}.

\begin{eqnarray*}
& &(\delta_* \xi) \wedge \eta \\
&=& \mu_* \left( \sum_{\sigma \in \perm_{(r_1 +1, r_2)}} (\sgn (\sigma)) \sigma \cdot \left( (\delta_* \xi) \times \eta \right)\right) \\
&=& -\mu_*  \sum_{\sigma \in \perm_{(r_1 +1, r_2)}} (\sgn (\sigma)) \sigma \cdot   \\ & &  \left( \sum_{\tau \in \perm_{(1, r_1)} }(\sgn (\tau)) \tau \cdot (x, \frac{1}{x}, t_1, \cdots, t_{r_1} )  \right) \times (y, t_{r_1 +1}, \cdots, t_{n} ) \\
&=& -\left( \sum_{\sigma \in \perm_{(r_1 + 1, r_2)}} (\sgn (\sigma)) \sigma \right) \cdot \\ & & \left( \sum_{\tau \in \perm_{(1, r_1)}} (\sgn (\tau)) (\tau \times {\rm Id}_{r_2})  \right) \cdot (xy, \frac{1}{x}, t_1, \cdots, t_n) \\
&=& - \left( \sum_{\nu \in \perm _{(1, r_1, r_2)}} (\sgn (\nu)) \nu \right) \cdot (xy, \frac{1}{x}, t_1, \cdots, t_n),
\end{eqnarray*}where the last equality follows from Lemma \ref{multiple shuffle 2}.

Similarly, 
\begin{eqnarray*}
& & (-1)^{r_1} \xi \wedge (\delta_* \eta) \\
&=& - \left( \sum_{\sigma \in \perm_{(r_1, r_2 +1)}} (\sgn (\sigma)) \sigma \right) \cdot \\ && \left( \sum_{\tau \in \perm_{(1, r_2)}} (\sgn (\tau)) ({\rm Id}_{r_1} \times \tau) \right) \cdot (xy, \frac{1}{y}, t_1, \cdots, t_n )\\
&=& - \left( \sum _{\nu \in \perm_{(1, r_1, r_2)}} (\sgn (\nu)) \nu \right) \cdot (xy, \frac{1}{y}, t_1, \cdots, t_n),
\end{eqnarray*}where the last equality follows from Lemma \ref{multiple shuffle 3}. Thus,

\begin{eqnarray*} & & \delta_*(\xi \wedge \eta) - (\delta_* \xi) \wedge \eta - (-1)^{r_1} \xi \wedge (\delta_* \eta)\\
&=& - \left( \sum_{\nu \in \perm _{(1, r_1, r_2)}} (\sgn (\nu))\nu \right) \cdot  \left( (xy, \frac{1}{x}) + (xy, \frac{1}{y}) - (xy, \frac{1}{xy})\right) \times (t_1, \cdots, t_n).
\end{eqnarray*}

On the other hand, 
\begin{eqnarray*}
& & - \partial (\xi \wedge' \eta) \\
&=& - \partial \left( \sum_{\nu \in \perm_{(1, r_1, r_2)}} (\sgn (\nu)) \nu \right) \cdot C_2 ^{\frac{1}{xy} , \left( \frac{1}{x}, \frac{1}{y} \right)} \times (t_1, \cdots, t_n).
\end{eqnarray*} Since $\partial_i ^j \left( C_2 ^{\frac{1}{xy}, \left( \frac{1}{x}, \frac{1}{y} \right)} \times (t_1, \cdots, t_n) \right) = 0$ for $i \geq 3$ and $j=0, \infty$, to prove the equality of the Proposition, we may assume that $r_1 = r_2 = n=0$, in which case the set $\perm_{(1, r_1, r_2)}$ is a singleton. Thus, it remains to check that
$$\left( xy, \frac{1}{x}\right) + \left( xy, \frac{1}{y}\right) - \left( xy, \frac{1}{xy} \right) = \partial C_2 ^{\frac{1}{xy}, \left( \frac{1}{x}, \frac{1}{y} \right)}.$$ But, this is indeed true by Lemma \ref{cycle C_2}. This finishes the proof.\end{proof}
\begin{corollary}The Connes boundary map $\delta$ induces $\delta_*$ which is a derivation for $\wedge$ in $ACH_* (k, *; A)$. 
\end{corollary}

Thus, we proved that
\begin{theorem}The triple $(ACH_* (k, *; A), \wedge, \delta_*)$ is a CDGA, where $\wedge$ and $\delta_*$ are induced by algebraic cycles. In particular, on $0$-cycles, the wedge product and the exterior derivation for $\Omega_{k/\bbZ}^*$ are motivic. In other words, the CDGA $\left( \Omega_{k/\mathbb{Z}} ^*, \wedge, d \right)$ is motivic.
\end{theorem}

\noindent \textbf{Acknowledgement} I would like to thank Donu Arapura, Spencer Bloch, C.-Y. Jean Chan, Alain Connes, William Heinzer,  Maxim Kontsevich, James McClure, and Bernd Ulrich for generously sharing their time during this work.

\bibliographystyle{amsplain}

\begin{thebibliography}{99}
\bibitem{Bl1} Bloch, S., \emph{Algebraic cycles and higher $K$-theory}, Adv. Math. 61 (3) (1986) 267--304.
\bibitem{Bl2} Bloch, S., http://www.math.uchicago.edu/\~{}bloch/cubical\_chow.pdf
\bibitem{BE1} Bloch, S. and Esnault, H., \emph{An additive version of higher Chow groups}, Ann. Scient. \'Ec. Norm. Sup. $4^e$ s\'erie, t. 36, (2003) 463--477.
\bibitem{BE2} Bloch, S. and Esnault, H., \emph{The additive dilogarithm}, Kazuya Kato's fiftieth birthday, Doc. Math. (2003) Extra Vol. 131--155.
\bibitem{C} Cathelineau, J.-L., \emph{Remarques sur les diff\'erentielle des polylogarithmes uniformes}, Ann. Inst. Fourier, Grenoble 46 (1996) no. 5, 1327--1347.
\bibitem{Cu} Cuvier, C., \emph{Alg\`ebres de Leibnitz: D\'efinitions, Propri\'et\'es}, Ann. Scient. \'Ec. Norm. Sup., $4^e$ s\'erie, t. 27, (1994) 1--45.
\bibitem{FT} Feigin, B., Tsygan, B., \emph{Additive $K$-theory} in $K$-theory, arithmetic and geometry,  (Moscow, 1984--1986), Lecture Notes in Math. 1289, Springer, Berlin (1987), 67--209.
\bibitem{G} Goncharov, A. B., \emph{Euclidean scissor congruence groups and mixed Tate motives over dual numbers}, Math. Res. Lett. 11 (2004), no. 5-6, 771--784.
\bibitem{KL} Krishna, A., Levine, M., \emph{Additive higher Chow groups of schemes}, arXiv:math.AG/0702138
\bibitem{LQ} Loday, J.-L., Quillen, D., \emph{Cyclic homology and the Lie algebra homology of matrices}, Comm. Math. Helv. 59 (1984), no. 4, 569--591.
\bibitem{L} Loday, J.-L., \emph{Cyclic homology} Appendix E by Mar\'ia O. Ronco. Second edition. Chapter 13 by the author in collaboration with Teimuraz Pirashvili. Grundlehren der Mathematischen Wissenschaften, 301. Springer-Verlag, Berlin, 1998. xx+513 pp.
\bibitem{L2} Loday, J.-L., \emph{Algebraic K-theory and the conjectural Leibniz K-theory}, K-theory 30 (2003) 105--127.
\bibitem{ML}  Mac Lane, S., \emph{Homology}, Reprint of the 1975 edition. Classics in Mathematics. Springer-Verlag, Berlin, 1995. x+422 pp.
\bibitem{NS} Nesterenko, Yu., Suslin, A., \emph{Homology of the full linear group over a local ring and Milnor's $K$-theory}, Izv. Akad. Nauk SSSR, Ser. Mat. 53, no. 1, (1989) 121-146 (Russian); English translation: Math. USSR, Izv. 34, no. 1 (1990) 121--145.
\bibitem{P1} Park, J., \emph{Regulators on additive higher Chow groups}, arXiv:math.AG/0605702v2 
\bibitem{P2} Park, J., \emph{Algebraic cycles and additive dilogarithm}, Int. Math. Res. Not. IMRN Vol 2007, no. 18, 19 pp. doi:10.1093/imrn/rnm067 
\bibitem{R} R\"ulling, K., \emph{The generalized de Rham-Witt complex over a field is a complex of zero-cycles}, J. Alg. Geom. 16 (2007) no. 1, 109--169
\bibitem{S} Suslin, A., \emph{Homology of $GL_n$, characteristic classes and Milnor $K$-theory}, Lecture Notes in Math. 1046, Springer, Berlin (1984) 357--375.
\bibitem{Sch} Schlessinger, M., \emph{Functor of Artin rings}, Trans. Amer. Math. Soc. 130 (1968) 208--222.
\bibitem{T} Totaro, B., \emph{Milnor $K$-theory is the simplest part of algebraic $K$-theory}, $K$-theory 6, no. 2 (1992) 177--189.

\end{thebibliography}

\end{document}